\documentclass[10pt]{article}
\usepackage{amsmath,graphicx,indentfirst,enumitem,amsthm,amssymb,tikz,color,geometry,float,ragged2e,mathtools,graphicx,url,blkarray,bigstrut,stmaryrd}
\usepackage[utf8]{inputenc}
\usepackage[subsection]{algorithm}
\usepackage[noend]{algpseudocode}
\usepackage{color}
\usepackage{amsfonts}
\usepackage{mathtools}
\usetikzlibrary{positioning,calc,arrows,automata}
\newtheorem{theorem}{THEOREM} 
\newtheorem{proposition}{PROPOSITION}

\theoremstyle{definition}
\newtheorem{definition}{DEFINITION}

\newtheorem{example}{EXAMPLE}[section]
\theoremstyle{plain}
\floatname{algorithm}{Algorithm}

\newcommand{\R}{\mathbb{R}}
\newcommand{\RightArrow}[1]{%
\parbox{#1}{\tikz{\draw[dashed,->](0,0)--(#1,0);}}
}
\newcommand{\LeftArrow}[1]{%
\parbox{#1}{\tikz{\draw[dashed,<-](0,0)--(#1,0);}}
}

\author{Omar De la Cruz Cabrera\thanks{Department of Mathematical Sciences, Kent State
University, Kent, OH 44242, USA.\hfill\break  Email: \texttt{odelacru@kent.edu} (O. De la Cruz Cabrera), 
\texttt{jjin3@kent.edu} (J. Jin), \texttt{reichel@math.kent.edu} (L. Reichel)}
\and
Jiafeng Jin\footnotemark[1] 
\and
Silvia Noschese\thanks{Dipartimento di Matematica, SAPIENZA Universit\`a di Roma, 
P.le Aldo Moro 5, 00185 Roma, Italy. \hfill\break Email: \texttt{noschese@mat.uniroma1.it}}
\and
Lothar Reichel\footnotemark[1]}

\title{Communication in Complex Networks}

\begin{document}
\setlength{\arrayrulewidth}{1pt}
\maketitle


\begin{abstract}
One of the properties of interest in the analysis of networks is \emph{global 
communicability}, i.e., how easy or difficult it is, generally, to reach nodes from other 
nodes by following edges. Different global communicability measures provide quantitative
assessments of this property, emphasizing different aspects of the problem.

This paper investigates the sensitivity of global measures of communicability to local changes.
In particular, for directed, weighted networks, we study how different global measures of
communicability change when the weight of a single edge is changed; or, in the unweighted
case, when an edge is added or removed.
The measures we study include the \emph{total network communicability}, based on the
matrix exponential of the adjacency matrix, and the \emph{Perron network communicability},
defined in terms of the Perron root of the adjacency matrix and the associated left and right 
eigenvectors.

Finding what local changes lead to the largest changes in global communicability has
many potential applications, including assessing the 
resilience of a system to failure or attack, guidance for incremental system improvements, 
and studying the sensitivity of global communicability measures to errors in the network
connection data.

\end{abstract}

\section{Introduction}\label{intro}
Many complex phenomena can be usefully modeled by networks. 
Mathematically, a 
network is represented by a graph, which consists of a set of vertices or nodes, and a set
of edges that connect pairs of vertices. Network models often simplify the representation 
of a complex system by disregarding some minutiae of reality, to make it feasible to use 
mathematical and computational methods of analysis; see, e.g.,
\cite{estrada2011structure,newman2010} for many examples. 

Sometimes, additional information
about the vertices and/or edges is indispensable for a fuller and more realistic 
understanding of a complex system. Examples include the use of \emph{weighted networks} 
\cite{barrat2004weighted,dlcmr,newman2004analysis}, in which edges between vertices are 
assigned different numerical values, so-called ``weights.'' In our setting, a higher 
weight for a given edge corresponds to a higher communication capacity between the nodes 
it connects.

An important characteristic of a network is how well communication can flow in it, i.e.,
how easy or difficult it is to reach one part of the network from another part by 
following edges. Several measures have been considered for quantifying communicability
on a global scale.
They include the \emph{diameter} of the graph that represents the network, the 
\emph{average distance} between nodes of this graph, and the communicability betweenness
of nodes; see Estrada et al. \cite{estrada2009communicability}. Information transfer 
between nodes also is 
studied with the aid of the thermal Green's function; here the network is considered
submerged in a thermal bath of some temperature $T$; see Estrada et al. 
\cite{estrada2021,estrada2011physics}. In this paper we concentrate on two communicability
measures: the 
\emph{total network communicability}, which was introduced by Benzi and Klymko 
\cite{benzi2013total}, and the \emph{Perron network communicability}, which we describe
below.

This article explores the sensitivity of the communicability measures mentioned to small,
local, changes of a network. Knowledge of the sensitivity can help answer several  
important questions about a network such as:
\begin{itemize}
\item 
How robust is a network to disturbances or attacks, and how can the network be modified to
be more robust?
\item 
Which edges of a network are very vulnerable, in the sense that the communicability 
decreases (relatively) significantly if these edges are removed? 
\item 
Can an addition of a new edge increase the communicability (relatively) significantly?
\item
Can a network be simplified by removing a few edges and retain essentially the same
communicability?
\item
How sensitive is the measured communicability of a network to incomplete information about 
the existing edges?
\end{itemize}

The graphs we consider may be unweighted or weighted. In an unweighted graph all edges have
the same weight, which we will choose to be one; in a weighted graph each edge has a 
positive weight.  
We are interested in which edge-weights to increase in order to increase the communicability 
of a graph the most, or which edges to add to or remove from a graph to achieve 
a significant increase or decrease, respectively, of the communicability. Our choice of which 
edge-weights to change, or which edges to add or remove, is based on the sensitivity of the 
communicability to changes in the edge-weight. We therefore investigate this sensitivity. 
Our approach is compared to some available approaches. Both undirected and directed 
graphs are considered.

We remark that graphs that have more than one connected component are deficient in their 
communicability, because not every node can communicate with every other node of the 
graph. 
Unless otherwise stated, we will assume the graphs under consideration to be 
connected. (For networks with more than one connected component, the results here can
be applied separately to each component.)

This paper is organized as follows: Section \ref{basicdef} defines basic concepts about 
graphs. Notions of communicability are reviewed and a new one is introduced in Section 
\ref{NOC}. Computed illustrations for some small graphs are presented in Section 
\ref{examples}. Section \ref{lsn} describes numerical methods for estimating the 
sensitivity of the total network communicability, and the sensitivity of the Perron network 
communicability to changes in the weights, for large-scale networks. Section 
\ref{largeexamples} presents a few computed results for large-scale networks. 
Concluding remarks are provided in Section \ref{conclusion}.

\section{Basic Definitions}\label{basicdef}
Networks are represented mathematically by graphs. The basic theory of graphs can
be found in many textbooks; see, e.g., \cite{estrada2011structure,newman2010} for
introductions focused on applications to the study of networks,
and \cite{biggs1993} for a deeper discussion of the matrices associated to a graph.
Below we will briefly state the definitions we need in order to fix the notation.

A \emph{weighted graph} 
$\mathcal{G}=\langle\mathcal{V},\mathcal{E},\mathcal{W}\rangle$ consists of a set of 
\emph{vertices} or \emph{nodes} $\mathcal{V}=\{v_1,v_2,\dots,v_n\}$, a set of \emph{edges} 
$\mathcal{E}=\{e_1,e_2,\dots,e_m\}$ that connect the nodes, and a 
map $\mathcal{W}$ that assigns to each edge a \emph{weight}, which for the purposes of
this article always will be a positive real number.
An edge $e_k$ is said to be \emph{directed} if it 
starts at a vertex $v_i$ and ends at a vertex $v_j$. This edge is denoted by
$e(v_i\rightarrow v_j)$ and has the associated \emph{weight} $w_{ij}>0$. If there also is 
an edge $e(v_j\rightarrow v_i)$ with the same weight $w_{ji}=w_{ij}$, then we may identify
the directed edges $e(v_i\rightarrow v_j)$ and $e(v_j\rightarrow v_i)$ with an 
\emph{undirected} edge with weight $w_{ij}$; we denote undirected edges by 
$e(v_i\leftrightarrow v_j)$. A graph with only
undirected edges is said to be \emph{undirected}; otherwise the graph is \emph{directed}. 
In this work, we will consider only graphs without multiple edges or self-loops.
The \emph{adjacency matrix} for a graph $\mathcal{G}$ is the
matrix $A=[w_{ij}]_{i,j=1}^n\in\R^{n\times n}$, whose entries are the edge-weights; if 
there is no edge $e(v_i\rightarrow v_j)$ in $\mathcal{G}$, then $w_{ij}=0$. For an 
\emph{unweighted graph}, all positive entries $w_{ij}$ of $A$ equal one. When 
$\mathcal{G}$ is undirected, then $A$ is symmetric.

A sequence of edges (not necessarily distinct) such that
$\{e(v_1\rightarrow v_2),e(v_2\rightarrow v_3),\ldots,e(v_k\rightarrow v_{k+1})\}$ form a
\emph{walk} of length $k$. If $v_{k+1}=v_1$, then the walk is said to be \emph{closed}. For 
further discussions on networks and graphs; see \cite{estrada2011structure,newman2010}.

\section{Notions of Communicability}\label{NOC}
There are many measures of communicability of a network. For instance, the diameter of the
graph that represents a network provides a measure of how easy it is for the nodes of the 
graph to communicate. We recall that for an unweighted graph, its diameter is the maximal 
length of the shortest path between any pair of distinct nodes of the graph. The 
definition has to be adjusted for weighted graphs. 
We will not use the diameter in this paper, because as a ``worst case'' measure it is fairly
crude. For instance, let ${\mathcal G}$ be an unweighted complete graph with 
$n\geq 4$ nodes from which one edge is removed. Then the diameter of ${\mathcal G}$ 
is $2$. More edges can be removed so that the diameter remains $2$. In fact, one can 
remove any $1$ or $2$ edges from a complete graph with $4$ nodes and then obtains a 
graph with diameter $2$.

This paper focuses on the total network communicability, which is defined with the aid of 
the exponential of the adjacency matrix, and on the Perron network communicability, which 
is defined with the Perron root and the right and left Perron vectors of the adjacency 
matrix. This section discusses these communicability measures and their sensitivity 
to changes in the weights that define the adjacency matrix. Small examples that illustrate
the performance of these measures are presented.

\subsection{The modified matrix exponential and network communicability}\label{MENC}
Consider an unweighted graph $\mathcal{G}$ with adjacency matrix $A\in\R^{n\times n}$. 
Then the $(ij)^{th}$ entry of the matrix $A^k$ counts the number of walks of length $k$
between the vertices $v_i$ and $v_j$. For weighted graphs, the interpretation of the 
$(ij)^{th}$ entry of the matrix $A^k$ has to be modified. A \emph{matrix function} that is
analytic at the origin and vanishes there can be defined by a formal Maclaurin series
\begin{equation}\label{matfun}
f(A)=\sum_{k=1}^{\infty}c_kA^k.
\end{equation}
For the moment  we ignore the convergence properties of this series. Long 
walks are usually considered less important than short walks, because information flows more 
easily through short walks than through long ones. Therefore, matrix functions applied in 
network analysis generally have the property that $0\leq c_{k+1}\leq c_k$ for all 
sufficiently large $k$.
The possibly most common matrix function used in network analysis is the matrix 
exponential; see \cite{estrada2011structure,estrada2005subgraph} for discussions and 
illustrations. We prefer to use the modified matrix exponential
\begin{equation}\label{modexp}
\exp_0(A):=\exp(A)-I, 
\end{equation}
where $I$ denotes the identity matrix, because the first term in the Maclaurin series of 
$\exp(A)$ has no natural interpretation in the context of network modeling. For the 
modified matrix exponential, we have $c_k=1/k!$ for $k\geq 1$, and the series 
\eqref{matfun} converges for any adjacency matrix $A$. The quantity $[\exp_0(A)]_{ii}$ is
commonly referred to as the \emph{subgraph centrality} of the vertex $v_i$; it measures 
the ease of leaving node $v_i$ and returning to this node by following the edges of the 
graph; see \cite{estrada2011structure,estrada2005subgraph}, though we remark that these 
references apply the matrix exponential instead of \eqref{modexp}. The subgraph centrality 
is an appropriate measure for undirected graphs; a discussion about directed graphs is 
provided in \cite{dlcmr1}.

The communicability between distinct vertices $v_i$ and $v_j$, $i\ne j$, is defined by
\[
[\exp_0(A)]_{ij}=\sum_{k=1}^{\infty}\frac{[A^k]_{ij}}{k!};
\]
see \cite{estrada2011structure,estrada2008} for the analogous definition based on 
$\exp(A)$. It 
accounts for all the possible routes of communication between the vertices $v_i$ and $v_j$
in the network defined by the adjacency matrix $A$, and assigns more weight to shorter 
routes than to longer ones. The larger the value of $[\exp_0(A)]_{ij}$, the better is the 
communicability between the vertices $v_i$ and $v_j$. 

We measure how effectively information can be transmitted across the whole network by the 
\emph{total network communicability}, 
\begin{equation}\label{eq1}
C^{\rm TN}(A)=\mathbf{1}^T \exp_0(A) \mathbf{1}, 
\end{equation}
where $\mathbf{1}=[1,1,\ldots,1]^T\in\R^n$. Benzi and Klymko \cite{benzi2013total} defined 
the total network communicability with $\exp(A)$; this yields values that are $n$ larger 
than the values obtained with \eqref{eq1}.

We are interested in determining the sensitivity of the total network communicability to 
changes in the weight $w_{ij}$ of the edge $e(v_i\rightarrow v_j)$. Therefore, we compute 
the partial derivative of $C^{\rm TN}(A)$ with respect to $w_{ij}$. For that, we introduce the 
Fr\'{e}chet derivative $L(A,E_{ij})$ of the modified matrix exponential $\exp_0(A)$ with 
respect to the direction $E_{ij}=\mathbf{e}_i\mathbf{e}_j^T$, where 
$\mathbf{e}_k=[0,\ldots,0,1,0,\ldots,0]^T\in\R^n$ denotes the $k$th axis vector, given by
\begin{equation}\label{eq2a} 
L(A,E_{ij})= \lim_{t\rightarrow0}\frac{\exp_0(A+tE_{ij})-\exp_0(A)}{t};
\end{equation}
see, e.g., \cite{higham2008functions,OR}. Then
\[
\frac{\partial C^{\rm TN}(A)}{\partial w_{ij}} = 
\lim_{t\rightarrow0}\frac{C^{\rm TN}(A+tE_{ij})-C^{\rm TN}(A)}{t} =
 \mathbf{1}^T \cdot L(A,E_{ij})\cdot \mathbf{1}
\]
shows the rate of change of the total network communicability between the vertices $v_i$ 
and $v_j$ in direction $E_{ij}$ due to a change in the edge-weight $w_{ij}$. 

\begin{definition}\label{def1}
Let $\mathcal{G}=\{\mathcal{V},\mathcal{E},\mathcal{W}\}$ be a graph with
adjacency matrix $A=[w_{ij}]\in\R^{n\times n}$, where $w_{ij}>0$ if there is an edge 
$e(v_i\rightarrow v_j)$ in $\mathcal{G}$, and $w_{ij}=0$ otherwise. We define the 
\emph{total network sensitivity with respect to the weight} $w_{ij}$ as
\begin{equation}\label{eq2}
S^{\rm TN}_{ij}(A)= \mathbf{1}^T \cdot L(A,E_{ij})\cdot \mathbf{1},
\end{equation}
as well as the \emph{total network sensitivity} as
\[
S^{\rm TN}(A)=\sum_{i,j=1}^n S^{\rm TN}_{ij}(A).
\]
\end{definition}

The total network sensitivity with respect to the weight $w_{ij}$ shows the rate of change
of $C^{\rm TN}(A)$ with respect to a change in the edge-weight $w_{ij}$. The following 
result follows from Higham \cite[Theorem 3.6]{higham2008functions}.

\begin{proposition}\label{thm1}
Let f be $2n-1$ times continuously differentiable in a connected open set $\Omega$ in the 
complex plane containing the origin, and assume that the spectrum of the matrix 
$A\in\R^{n\times n}$ is in $\Omega$. Then the Fr\'{e}chet derivative $L(A,E_{ij})$ exists
and satisfies 
\begin{equation}\label{thm1result}
f\left({\begin{bmatrix}
A & E_{ij} \\
0 & A \\
\end{bmatrix}}\right)
=\begin{bmatrix}
f(A) & L(A,E_{ij}) \\
0 & f(A)\\
\end{bmatrix}.
\end{equation}
\end{proposition}

We are interested in the special case of Proposition \ref{thm1} when $f(t)=\exp(t)-1$. Then 
the Fr\'{e}chet derivative $L(A,E_{ij})$ in \eqref{thm1result} is defined by \eqref{eq2a}. 
One has 
\begin{equation}\label{defL}
\exp_0(A+t E_{ij}) - \exp_0(A) - t L(A, E_{ij})= O(t^2)\mbox{  as  }
t\rightarrow 0.
\end{equation}

\subsection{Perron network communicability}\label{MP}
Let $A=[w_{ij}]\in\R^{n\times n}$ be a nonnegative irreducible adjacency matrix for a 
graph and let $\rho$ be its Perron root (see, e.g., \cite{horn1985matrix} for a full
treatment of the Perron--Frobenius Theorem). Then there are unique right and left real 
eigenvectors $\mathbf{x}=[x_1,x_2,\ldots,x_n]^T\in\R^n$ and 
$\mathbf{y}=[y_1,y_2,\ldots,y_n]^T\in\R^n$, respectively, of unit Euclidean norm with 
positive entries associated with $\rho$, i.e., 
\begin{equation}\label{pvects}
A\mathbf{x}=\rho\mathbf{x},\qquad \mathbf{y}^TA=\rho\mathbf{y}^T.
\end{equation}
We recall that a node $v_i$ is referred to as a \emph{sink} if there are no edges pointing
from this node to any other node, and a node $v_i$ is said to be a \emph{source} when 
there are no edges from other nodes pointing to $v_i$. Undirected edges should be 
considered as ``two-way streets'' in this context. A directed graph is said to be 
\emph{strongly connected} if every vertex $v_i$ can be reached by any other vertex $v_j$,
$j\ne i$, by following edges in their direction. It is well known that an adjacency matrix
is irreducible if and only if the associated graph is strongly connected; see, e.g.,
\cite{horn1985matrix}. In particular, a graph with a sink or source is not strongly 
connected, and the adjacency matrix $A$ for such a graph is reducible. A common
approach to obtain a nearby irreducible adjacency matrix is to make a small positive 
perturbation of every entry of $A$. For instance, we may replace $A$ by the irreducible 
adjacency matrix $\hat{A}=A+\delta A\in\R^{n\times n}$, where $\delta A$ is a small multiple of 
the matrix $\mathbf{1}\mathbf{1}^T$ and $\mathbf{1}=[1,1,\ldots,1]^T\in\R^n$. This makes 
the matrix $\hat{A}=A+\delta A$ irreducible with $\|\hat{A}\|_2\approx\|A\|_2$. Here and 
throughout this paper, $\|\cdot\|_2$ denotes the spectral matrix norm or the Euclidean 
vector norm. In this section we will assume that, in case the given adjacency matrix $A$ 
is reducible, it is modified in this manner.  We may therefore assume that the right and 
left Perron vectors are unique up to scaling.
 
Let $C\in\R^{n\times n}$ be a nonnegative matrix such that $\| C \|_2=1$, and let 
$\varepsilon>0$ be small. Denote the Perron root of 
$A+\varepsilon C$ by $\rho+\delta\rho$. Then
\[
\delta\rho=\varepsilon\frac{\mathbf{y}^TC\mathbf{x}}{\mathbf{y}^T\mathbf{x}}+
O(\varepsilon^2);
\]
see \cite{milanese2010approximating}. Moreover,
\begin{equation}\label{condrho}
\frac{\mathbf{y}^T C \mathbf{x}} {\mathbf{y}^T\mathbf{x}}=
\frac{|\mathbf{y}^T C \mathbf{x}|} {\mathbf{y}^T\mathbf{x}}\leq
\frac{\|\mathbf{y}\|_2\|C\|_2\|\mathbf{x}\|_2} {\mathbf{y}^T\mathbf{x}}=
\frac{1}{\cos\theta},
\end{equation}
where $\theta$ is the angle between $\mathbf{x}$ and $\mathbf{y}$. The quantity 
$1/\cos\theta$ is referred to as the \emph{condition number} of $\rho$ and denoted by 
$\kappa(\rho)$; see Wilkinson \cite[Chapter 2]{Wi}. Equality is attained in 
\eqref{condrho} for $C=\mathbf{y}\mathbf{x}^T$. 

Consider increasing the entry $w_{ij}$, $i\ne j$, of $A$ slightly by $\varepsilon>0$. This 
corresponds to increasing the weight $w_{ij}$ of an existing edge $e(v_i\rightarrow v_j)$ 
by $\varepsilon$, or to introducing a new edge $e(v_i\rightarrow v_j)$ with weight 
$\varepsilon$. The corresponding matrix $C$ is given by 
\begin{equation}\label{Cij}
C=E_{ij}=\mathbf{e}_i\mathbf{e}_j^T.
\end{equation}
The impact on the Perron root of the change $\varepsilon C$ of $A$ is 
\[
\delta\rho=\varepsilon\frac{y_ix_j}{\mathbf{y}^T\mathbf{x}}+O(\varepsilon^2).
\]

We are interested in choosing a perturbation $\varepsilon C$ of $A$ with $C$ of the form 
\eqref{Cij}, so that $\rho$ is increased as much as possible. The above relation shows 
that we should choose $i$ and $j$ so that
\[
x_j=\max_{1\leq k\leq n}x_k, \qquad y_i=\max_{1\leq k\leq n}y_k.
\]
Recall that $\mathbf{x}$ and $\mathbf{y}$ are unit vectors. Therefore, $x_j=1$ and $y_i=1$
implies that $x_k=0$ for all $k\ne j$ and $y_k=0$ for all $k\ne i$. Then
$C=E_{ij}=\mathbf{y}\mathbf{x}^T$ and the maximum perturbation of the Perron root is 
$\delta\rho=\varepsilon\kappa(\rho)$. 

We define the \emph{Perron root sensitivity} with respect to the direction
$E_{ij}=\mathbf{e}_i\mathbf{e}_j^T$ as 
\[
S^{\rm PR}_{ij}(A)=\frac{y_ix_j}{\mathbf{y}^T\mathbf{x}},
\]
as well as the \emph{Perron root sensitivity matrix}
\[
S^{\rm PR}(A)= \bigg[S^{\rm PR}_{ij}(A)\bigg]_{i,j=1}^n=
\frac{\mathbf{y}\mathbf{x}^T}{\mathbf{y}^T\mathbf{x}}\in\R^{n\times n}.
\]
Finally, we introduce the \emph{Perron network communicability}, 
\begin{equation}\label{CPN}
C^{\rm PN}(A)=\exp_0(\rho)\,\mathbf{1}^T\mathbf{x}\mathbf{y}^T\mathbf{1}=
\exp_0(\rho)\left(\sum_{j=1}^nx_j\right)\left(\sum_{j=1}^ny_j\right),
\end{equation}
which is analogous to the total network communicability \eqref{eq1}, but is easier to 
compute for a large adjacency matrix $A$. Let $\|\cdot\|_1$ denote the vector $1$-norm. 
Since 
\[
\sum_{j=1}^nx_j=\|\mathbf{x}\|_1\leq n^{1/2}\|\mathbf{x}\|_2=n^{1/2},\qquad
\sum_{j=1}^ny_j=\|\mathbf{y}\|_1\leq n^{1/2}\|\mathbf{y}\|_2=n^{1/2},
\]
we have the bound
\[
C^{\rm PN}(A) \leq n \exp_0 (\rho).
\]
In general, we expect the Perron network communicability to increase the most when 
increasing the edge-weight that makes the Perron root change the most. 

An alternative way to study the sensitivity of the Perron network communicability is to 
determine the first-order partial derivative of $C^{\rm  PN}(A)$ with respect to $w_{ij}$.
Introduce the Fr\'{e}chet derivative $L^{\rm PN}(A,E_{ij})$ of the matrix function 
$\exp_0(\rho)\mathbf{x}\mathbf{y}^T$ with respect to the direction 
$E_{ij}=\mathbf{e}_i\mathbf{e}_j^T$, i.e.,
\[
L^{\rm PN}(A,E_{ij})=\lim_{t\rightarrow0}
  \frac{\exp_0(\hat{\rho}(t))\hat{\mathbf{x}}(t)\hat{\mathbf{y}}(t)^T-
  \exp_0(\rho)\mathbf{x}\mathbf{y}^T}{t},
\]
where $\rho$, $\mathbf{x}$, and $\mathbf{y}$ are the Perron root, and the right and left 
Perron vectors of $A$, respectively, and $\hat{\rho}(t)$, $\hat{\mathbf{x}}(t)$, and 
$\hat{\mathbf{y}}(t)$ are the Perron root, and the right and left Perron vectors of 
$A+tE_{ij}$. Then
\[
\frac{\partial C^{\rm PN}(A)}{\partial w_{ij}} =  
\lim_{t\rightarrow0}\frac{C^{\rm PN}(A+tE_{ij})-C^{\rm PN}(A)}{t} =
\mathbf{1}^T \cdot L^{\rm PN}(A,E_{ij})\cdot \mathbf{1}
\]
is the rate of change of the Perron network communicability between the vertices $v_i$ and
$v_j$ in the direction $E_{ij}$ due to a change in $w_{ij}$. For the examples shown in the
following sections, we choose $t=2\cdot10^{-5}$ when computing $L^{\rm PN}(A,E_{ij})$ unless
explicitly stated otherwise.

\begin{definition}\label{def2}
Let $\mathcal{G}=\{\mathcal{V},\mathcal{E},\mathcal{W}\}$ be a graph with
adjacency matrix $A=[w_{ij}]\in\R^{n\times n}$, where $w_{ij}>0$ if there is an edge 
$e(v_i\rightarrow v_j)$ in $\mathcal{G}$, and $w_{ij}=0$ otherwise. We define the 
\emph{Perron network sensitivity with respect to the weight} $w_{ij}$ as
\begin{equation}\label{peq2}
S^{\rm PN}_{ij}(A)= \mathbf{1}^T \cdot L^{\rm PN}(A,E_{ij})\cdot \mathbf{1},
\end{equation}
as well as the \emph{Perron network sensitivity} as
\[
S^{\rm PN}(A)=\sum_{i,j=1}^n S^{\rm PN}_{ij}(A).
\]
\end{definition}

The Perron network sensitivity with respect to the weight $w_{ij}$ shows the rate of 
change of $C^{\rm PN}(A)$ with respect to a change in the edge-weight $w_{ij}$. We are 
interested in investigating how $C^{\rm PN}(A)$ relates to $C^{\rm TN}(A)$. Assume first 
that the spectral factorizations
\begin{equation}\label{spf}
A=X\Lambda X^{-1},\quad A^T=\widetilde{Y}\Lambda\widetilde{Y}^{-1}
\end{equation}
exist, where $\Lambda={\rm diag}[\rho,\lambda_2,\ldots,\lambda_n]$ and the columns of the
eigenvector matrix $X=[\mathbf{x},\mathbf{x}_2,\ldots,\mathbf{x}_n]$ are scaled to be of
unit Euclidean norm. We remark that matrices with a spectral factorization are dense among
all matrices in $\R^{n\times n}$. We may choose $\widetilde{Y}=
[\widetilde{\mathbf{y}},\widetilde{\mathbf{y}}_2,\ldots,\widetilde{\mathbf{y}}_n]=X^{-T}$.
Then $A=X\Lambda\widetilde{Y}^T$. 
The vector $\widetilde{\mathbf{y}}$ is a rescaling of the left Perron vector $\mathbf{y}$
in \eqref{pvects}. This normalization of the columns of $\widetilde{Y}$ yields 
$\widetilde{\mathbf{y}}^T\mathbf{x}=1$ and $\widetilde{\mathbf{y}}_j^T\mathbf{x}_j=1$
for $j=2,3,\ldots,n$. In particular, this implies
\[
1=\widetilde{\mathbf{y}}^T\mathbf{x} =\|\mathbf{x}\| \|\widetilde{\mathbf{y}}\| \cos \theta, 
\quad\mbox{i.e.,}\quad \|\widetilde{\mathbf{y}}\|=\frac{1}{\cos\theta}=\kappa(\rho).
\]
Assume that $\rho$ is significantly larger than $|\lambda_j|$ for $j=2,3,\ldots,n$. Then
\begin{eqnarray} 
\nonumber
C^{\rm TN}(A)&=&\mathbf{1}^TX\exp_0(\Lambda)\widetilde{Y}^T\mathbf{1}\\
\nonumber
&=& \exp_0(\rho)\mathbf{1}^T\mathbf{x}\widetilde{\mathbf{y}}^T\mathbf{1}+
\sum_{j=2}^n \exp_0(\lambda_j)\mathbf{1}^T\mathbf{x}_j\widetilde{\mathbf{y}}_j^T\mathbf{1}\\
\label{CTNA}
&=& \kappa(\rho)C^{\rm PN}(A)+ 
\sum_{j=2}^n \exp_0(\lambda_j)\mathbf{1}^T\mathbf{x}_j\widetilde{\mathbf{y}}_j^T\mathbf{1}\\
\nonumber
&\approx&\kappa(\rho)C^{\rm PN}(A).
\end{eqnarray}
Thus, the total network communicability depends on the conditioning of the Perron root. 
For a general matrix $A$, we can use the bounds $|\mathbf{1}^T\mathbf{x}_j|\leq n^{1/2}$ 
and $|\widetilde{\mathbf{y}}_j^T\mathbf{1}|\leq\|\widetilde{\mathbf{y}}_j\|_1$ in 
\eqref{CTNA}. When the graph that defines $A$ is undirected, the matrix $A$ is symmetric, 
and we can let $\widetilde{Y}=X$ be orthogonal. In this case, we have $\kappa(\rho)=1$ and
can use 
\[
|\mathbf{1}^T\mathbf{x}_j\widetilde{\mathbf{y}}_j^T\mathbf{1}|\leq n,\qquad j=2,3,\ldots,n,
\]
in \eqref{CTNA}.

If the matrix $A$ does not have the factorizations \eqref{spf}, then an analogous argument
can be made using the Jordan canonical form, where we use the fact that the right and left 
Perron vectors exist and are unique up to scaling also in this situation.

We finally remark that both the computation of the total network sensitivity 
$S^{\rm TN}(A)$ and of the Perron network sensitivity $S^{\rm PN}(A)$ for a graph with
$n$ nodes requires the evaluation of $n^2$ Fr\'echet derivatives. This makes the 
evaluation of these quantities expensive for networks with many nodes. The evaluation
of the Perron root sensitivity matrix $S^{\rm PR}(A)$ typically is much cheaper for 
large networks. We will return to this issue below.

\subsection{Examples with small networks}\label{examples}
This section describes a few small examples that illustrate the use of Fr\'echet
derivatives  
and the Perron root sensitivity. The computational effort required to 
compute all Fr\'echet derivatives for large-scale networks can be significant. How to 
reduce the computational effort for large-scale networks is discussed in Section \ref{lsn}.

In the first example, which is a small weighted graph, we compute the total network 
sensitivity and the Perron network sensitivity with respect to the weights $w_{ij}$ for 
$i \neq j$ to decide which edge-weight should be increased to enhance the total network 
communicability or the Perron network communicability as much as possible. We study 
the Perron root sensitivity by computing the left and right Perron vectors of $A$. The 
latter computations suggest which weight should be increased. Our second example 
differs from the first one in that the graph is unweighted and directed. 

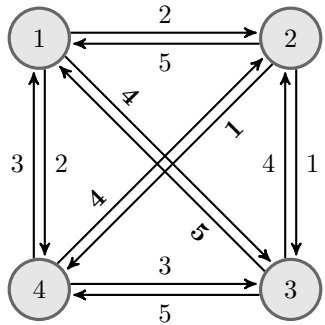
\begin{figure}[H]
\begin{center}
\begin{tikzpicture}[
 roundnode/.style={circle, draw=black!60, fill=black!10, very thick, minimum size=8mm},
 ->,>=stealth',shorten >=1pt,node distance=2.5cm,auto,thick]

\node[roundnode]      (node1)                                                   {$1$};
\node[roundnode]      (node2)       [right=2.5cm of node1]           {$2$};
\node[roundnode]      (node3)       [below=2.5cm of node2]         {$3$};
\node[roundnode]      (node4)       [below=2.5cm of node1]         {$4$};

\draw[->] (node1.10) --node[above] {2}  (node2.170);
\draw[->] (node2.-170) --node[below] {5} (node1.-10);
\draw[->] (node1.-35) --node[above, sloped, near start] {\textbf{4}}  (node3.125);
\draw[->] (node3.145) --node[below, sloped, near start] {\textbf{5}}  (node1.-55);
\draw[->] (node1.-80) --node[right] {2}  (node4.80);
\draw[->] (node4.100) --node[left] {3}  (node1.-100);
\draw[->] (node2.-80) --node[right] {1}  (node3.80);
\draw[->] (node3.100) --node[left] {4}  (node2.-100);
\draw[->] (node4.55) --node[above, sloped, near start] {\textbf{4}}  (node2.-145);
\draw[->] (node2.-125) --node[below, sloped, near start] {\textbf{1}}  (node4.35);
\draw[->] (node4.10) --node[above] {3}  (node3.170);
\draw[->] (node3.-170) --node[below] {5}  (node4.-10);
\end{tikzpicture}
\caption{Graph of Example 3.1. The edge-weights are marked next to the edges.}
\label{fig1}
\end{center}
\end{figure}

\begin{example}\label{ex1}
Consider the weighted graph of Figure \ref{fig1} with associated adjacency matrix 
\begin{equation}\label{adj1}
A=\begin{bmatrix} 
0 & 2 & 4 & 2  \\ 
5 & 0 & 1 & 1  \\ 
5 & 4 & 0 & 5  \\ 
3 & 4 & 3 & 0 
\end{bmatrix}.
\end{equation}
The corresponding Perron root sensitivity matrix of $A$ is
\begin{equation}\label{PRSmat}
S^{\rm PR}(A)=\begin{bmatrix}
0.2956 & 0.2339 & 0.4241 & 0.3250 \\
0.2336 & 0.1848 & 0.3352 & 0.2568 \\
0.2109 & 0.1669 & 0.3026 & 0.2319 \\
0.1973 & 0.1562 & 0.2832 & 0.2170 \\
\end{bmatrix},
\end{equation}
which shows that the Perron root sensitivity $S^{\rm PR}_{ij}$ is maximized for 
$\{i,j\}=\{1,3\}$. This indicates that the Perron root is increased the most when 
increasing the weight $w_{13}$ of the edge  $e(v_1\rightarrow v_3)$. We expect the Perron
network communicability to increase the most when increasing the edge-weight that makes 
the Perron root change the most. Table \ref{table1} confirms this. The table displays 
the total network sensitivity and the Perron network sensitivity with respect to changes 
in the weights $w_{ij}$ for $i \neq j$, as well as the total network communicability and 
the Perron network communicability of the graph obtained when increasing each weight 
$w_{ij}$, $i\ne j$, by one. The table shows the total network communicability to increase 
the most by increasing the weight of the edge with the largest total network sensitivity 
$S^{\rm TN}_{ij}$, i.e., weight $w_{13}$.  Increasing this weight, which also is 
associated with the largest Perron network sensitivity $S^{\rm PN}_{13}$, gives the 
largest Perron network communicability, $C^{\rm PN}(A+E_{13})$. Thus, if the edges 
represent roads, the weights represent the width of each road, and we would like to
increase the communicability the most by widening one road, then we should widen the road 
represented by the edge $e(v_1\rightarrow v_3)$. 

\begin{table}[H]
\renewcommand{\arraystretch}{1.15}
\begin{center}
\begin{tabular}{cccccc}
\hline  
$\{i,j\}$ &   $S^{\rm TN}_{ij}$   &   $C^{\rm TN}(A+E_{ij})$    &   $\{i,j\}$    & $S^{\rm PN}_{ij}$ &   $C^{\rm PN}(A+E_{ij})$ \\
\hline
$\{1,3\}$ & 22615  & 82269   & $\{1,3\}$      &  22781  &   79872    \\
$\{2,3\}$ & 18221  & 76339   & $\{2,3\}$      &  18247  &   73711    \\
$\{1,4\}$ & 17662  & 75511   & $\{1,4\}$      &  17577  &   72722    \\
$\{4,3\}$ & 15611  & 73124   & $\{4,3\}$      &  15009  &   69720    \\
$\{2,4\}$ & 14225  & 71324   & $\{2,4\}$      &  14078  &   68481   \\
$\{1,2\}$ & 13151  & 70097   & $\{1,2\}$      &  12411  &   66543    \\
$\{2,1\}$ & 12957  & 69606   & $\{2,1\}$      &  12394  &   66250   \\
$\{3,4\}$ & 12883  & 69588   & $\{3,4\}$      &  12134  &   66022   \\
$\{3,1\}$ & 11734  & 68303   & $\{3,1\}$      &  10666  &   64389   \\
$\{4,1\}$ & 11098  & 67434   & $\{4,1\}$       &  10188  &   63702   \\
$\{3,2\}$ & 9585    & 65627   & $\{3,2\}$       &  8562    &   61789   \\
$\{4,2\}$ & 9063    & 65011   & $\{4,2\}$       &  8176    &   61329   \\
\hline
\end{tabular}\\
\vspace{0.05in}
\end{center} 
\caption{Example \ref{ex1}: The total network sensitivity $S^{\rm TN}_{ij}$ and the Perron
network sensitivity $S^{\rm PN}_{ij}$ with respect to changes in the weight $w_{ij}$, 
for $i\neq j$, along with  the total network communicability and the Perron network communicability 
when the weight $w_{ij}$ of the edge $e(v_i\rightarrow v_j)$ is increased by one.}
\label{table1}
\end{table}

The total network sensitivities $S^{\rm TN}_{ij}$ and the Perron network sensitivities 
$S^{\rm PN}_{ij}$ with respect to changes in the weights $w_{ij}$ of Table \ref{table1} 
also can be used to assess which weight(s) to decrease to reduce the total network 
communicability or the Perron network communicability of the network of Figure \ref{fig1} 
the most. The fact that $S^{\rm TN}_{13}$ and $S^{\rm PN}_{13}$ are the largest sensitivities suggests 
that we should reduce the weight $w_{13}$ to reduce the total network communicability and 
the Perron network communicability the most. Furthermore, the Perron root sensitivity 
matrix \eqref{PRSmat} suggests that both the Perron network communicability \eqref{CPN} and the 
Perron root will decrease the most when decreasing the weight $w_{13}$. Indeed, tabulating 
$C^{\rm TN}(A-E_{ij})$ and $C^{\rm PN}(A-E_{ij})$ for $1\leq i,j\leq 4$, $i\ne j$, shows 
$C^{\rm TN}(A-E_{13})$ and $C^{\rm PN}(A-E_{13})$ to be minimal. $~~\Box$
\end{example}

In the adjacency matrix \eqref{adj1} all off-diagonal entries are positive. This is not
important for the approach described. We can compute the total network sensitivity 
\eqref{eq2} and the Perron network sensitivity \eqref{peq2} independently of the values of the weights
$w_{ij}$. If $S^{\rm TN}_{ij}$ or $S^{\rm PN}_{ij}$ is the largest sensitivity and $w_{ij}=0$, 
then this indicates that the total network communicability or the Perron network 
communicability may be increased the most by adding the edge $e(v_i\rightarrow v_j)$ to 
the graph.

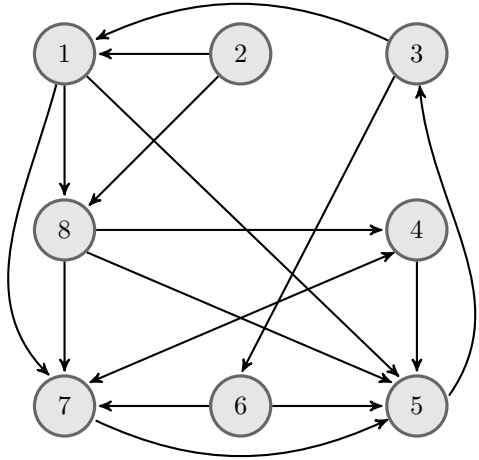
\begin{figure}[H]
\begin{center}
\begin{tikzpicture}[
 roundnode/.style={circle, draw=black!60, fill=black!10, very thick, minimum size=8mm},
 ->,>=stealth',shorten >=1pt,node distance=2.5cm,auto,thick]

\node[roundnode]      (node1)                                                      {$1$};
\node[roundnode]      (node2)       [right=1.5cm of node1]           {$2$};
\node[roundnode]      (node3)       [right=1.5cm of node2]         {$3$};
\node[roundnode]      (node4)       [below=1.5cm of node3]         {$4$};
\node[roundnode]      (node5)       [below=1.5cm of node4]           {$5$};
\node[roundnode]      (node6)       [left=1.5cm of node5]         {$6$};
\node[roundnode]      (node7)       [left=1.5cm of node6]         {$7$};
\node[roundnode]      (node8)       [below=1.5cm of node1]         {$8$};

\draw[->] (node1.-45) -- (node5.120);
\draw[->] (node1.-90) -- (node8.90);
\draw[->] (node1.-105) to [out=-105,in=135]  (node7.115);
\draw[->] (node2.180) -- (node1.0);
\draw[->] (node2.-135) -- (node8.45);
\draw[->] (node3.-135) -- (node6.90);
\draw[->] (node3.155) to [out=155,in=25] (node1.25);
\draw[<-] (node3.-85) to [out=-85,in=55]  (node5.15);
\draw[->] (node4.-90) -- (node5.90);
\draw[<->] (node4.-135) -- (node7.45);
\draw[<-] (node5.-155) to [out=-155,in=-25]  (node7.-25);
\draw[->] (node6.0) -- (node5.180);
\draw[->] (node6.180) -- (node7.0);
\draw[->] (node8.0) -- (node4.180);
\draw[->] (node8.-90) -- (node7.90);
\draw[->] (node8.-45) -- (node5.135);

\end{tikzpicture}
\caption{Graph of Example 3.2. All edges have unit weight.}
\label{fig_new}
\end{center}
\end{figure}

\begin{example} \label{ex_new}
Let $A\in\R^{8\times 8}$ be the adjacency matrix for the unweighted directed graph of 
Figure \ref{fig_new}. All its entries are either one of zero. The graph is not strongly 
connected, and therefore the matrix $A$ is reducible. To obtain an irreducible matrix 
$\hat{A}$, we add the perturbation matrix $\delta A=\delta \cdot\mathbf{1}\mathbf{1}^T$ to
$A$ for some $\delta>0$; thus, $\hat{A}=A+\delta A$. We choose the value of $\delta$ as
follows: Compute the Perron vectors for $\delta=10^{-4}$ and then reduce $\delta$ by a 
factor $10$ and determine new Perron vectors until the edge determined for two consecutive 
$\delta$-values is the same. For the present example, this gives $\delta=10^{-5}$. The 
matrix $\hat{A}=A+10^{-5}\cdot\mathbf{1}\mathbf{1}^T$ so determined is irreducible and the
right and left Perron vectors are unique up to scaling. We find the three largest entries 
of the Perron root sensitivity matrix $S^{\rm PR}(\hat{A})$ of $\hat{A}$ to be 
\[
S^{\rm PR}(\hat{A})_{5,1}=0.477305,\quad
S^{\rm PR}(\hat{A})_{5,2}=0.477298,\quad
S^{\rm PR}(\hat{A})_{5,8}=0.400601.
\]
This suggests that the Perron root may be increased the most by inserting the edge 
$e(v_5\rightarrow v_1)$ into the graph. Typically, the Perron network communicability 
is increased the most by increasing the weight for an edge (or inserting an edge) that 
results in the largest increase of the Perron root. However, as is illustrated by Table 
\ref{table_new}, this might not be the case when the largest Perron root sensitivities 
$S^{\rm PR}(\hat{A})_{i,j}$ are very close in size as in the present example, where 
$S^{\rm PR}_{5,1}$ and $S^{\rm PR}_{5,2}$ are very close. The table shows the
Perron network communicability to increase the most by adding the edge 
$e(v_5\rightarrow v_2)$. Table \ref{table_new} shows the top five total network sensitivities and Perron
network sensitivities with respect to perturbations in $w_{ij}$ for $i\neq j$, along with 
the total network communicabilities and Perron network communicabilities of the graph obtained when 
increasing each edge-weight $w_{ij}$ by one. The table shows the total network 
communicability to increase the most when adding the edge associated with the largest 
total network sensitivity $S^{\rm TN}_{ij}$. The Perron network communicability 
increases the most by inserting the edge associated with
the largest Perron network sensitivity $S^{\rm PN}_{ij}$.  $~~\Box$
\end{example}

\begin{table}[H]
\renewcommand{\arraystretch}{1.15}
\begin{center}
\begin{tabular}{cccccc}
\hline  
$\{i,j\}$ &   $S^{\rm TN}_{ij}$   &   $C^{\rm TN}(A+E_{ij})$    &$\{i,j\}$ &  $S^{\rm PN}_{ij}$ &   $C^{\rm PN}(\hat{A}+E_{ij})$ \\
\hline
$\{5,1\}$ & 16.8311  & 68.1499   & $\{5,2\}$ &  28.5369  &   58.3264    \\
$\{5,8\}$ & 16.0552  & 67.3050   & $\{5,1\}$ &  23.0099  &   56.9987    \\
$\{5,2\}$ & 14.9415  & 65.2404   & $\{5,3\}$ &  19.9219  &   51.2350    \\
$\{5,3\}$ & 14.1656  & 64.4248   & $\{5,8\}$ &  19.5778  &   53.8728   \\
$\{7,1\}$ & 13.2831  & 64.1815   & $\{7,2\}$ &  18.2576  &   50.5741   \\

\hline
\end{tabular}\\
\vspace{0.05in}
\end{center} 
\caption{Example \ref{ex_new}: The five largest total network sensitivities 
$S^{\rm TN}_{ij}$ and Perron network sensitivities $S^{\rm PN}_{ij}$ with respect to
perturbations in the weights $w_{ij}$, for $i\neq j$, along with the corresponding total 
network communicabilities and Perron network 
communicabilities when the edge-weight $w_{ij}$ is increased by one.}
\label{table_new}
\end{table}

For some graphs one can prove where to add an additional edge to maximize the total
network sensitivity. The following result provides an illustration. 

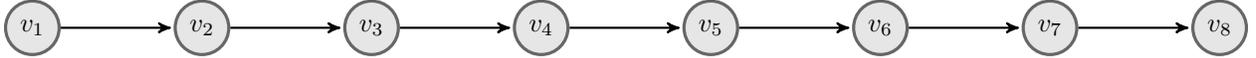
\begin{figure}[H]
\begin{center}
\begin{tikzpicture}[
roundnode/.style={circle, draw=black!60, fill=black!10, very thick, minimum size=6mm},
->,>=stealth',shorten >=1pt,node distance=2.5cm,auto,thick]

\node[roundnode]      (node1)                                                   {$v_1$};
\node[roundnode]      (node2)       [right=1.5cm of node1]      {$v_2$};
\node[roundnode]      (node3)       [right=1.5cm of node2]         {$v_3$};
\node[roundnode]      (node4)       [right=1.5cm of node3]         {$v_4$};
\node[roundnode]      (node5)       [right=1.5cm of node4]         {$v_5$};
\node[roundnode]      (node6)       [right=1.5cm of node5]         {$v_6$};
\node[roundnode]      (node7)       [right=1.5cm of node6]         {$v_7$};
\node[roundnode]      (node8)       [right=1.5cm of node7]         {$v_8$};

\draw[->] (node1.east) -- (node2.west);
\draw[->] (node2.east) -- (node3.west);
\draw[->] (node3.east) -- (node4.west);
\draw[->] (node4.east) -- (node5.west);
\draw[->] (node5.east) -- (node6.west);
\draw[->] (node6.east) -- (node7.west);
\draw[->] (node7.east) -- (node8.west);

\end{tikzpicture}
\caption{Graph of Theorem \ref{thm2}.}\label{fig3.0}
\end{center}
\end{figure}

\begin{theorem} \label{thm2}
Let $A=[w_{ij}]\in\R^{n\times n}$ be an adjacency matrix with all superdiagonal entries
equal to one and all other entries equal to zero. Figure \ref{fig3.0} displays the 
associated graph for $n=8$. Let $S_{ij}^{ \rm TN}(A)=\mathbf{1}^TL(A,E_{ij})\mathbf{1}$ be the
total network sensitivity with respect to the weight $w_{ij}$, where $L(A,E_{ij})$ is the 
Fr\'echet derivative of $A$ in the direction $E_{ij}$. Consider the addition of one edge to 
the graph defined by $A$. Then the total network sensitivity is maximized by inserting the
edge $e(v_n\rightarrow v_1)$. Thus,
\[
\max_{i,j}S_{ij}^{ \rm TN}(A)=S_{n1}^{ \rm TN}(A).
\]
\end{theorem}

\begin{proof}
Let $r,s$ be integers with $1\leq r,s\leq n$, and let 
$\mathbf{e}_k=[0,\ldots,0,1,0,\ldots,0]^T\in \R^n$ denotes the $k$th axis vector and define
$E_{rs}=\mathbf{e}_r\mathbf{e}_s^T$. Expanding $L(A,E_{rs})$ in terms of powers of $A$ 
yields
\begin{equation}\label{taylor}
S_{rs}^{ \rm TN}(A)=\mathbf{1}^TL(A,E_{rs})\mathbf{1}=
\mathbf{1}^T\Bigg(E_{rs}+\frac{AE_{rs}+E_{rs}A}{2!}+\frac{A^2E_{rs}+
AE_{rs}A+E_{rs}A^2}{3!}+\ldots\Bigg)\mathbf{1}.
\end{equation}
For notational convenience, denote the sum of the terms in the numerators that contain a 
total of $k$ powers of $A$ by $H_{rs}^{(k)}$, i.e., 
$H_{rs}^{(0)}=E_{rs}$, $H_{rs}^{(1)}=AE_{rs}+E_{rs}A$, etc. Then the right-hand side of 
(\ref{taylor}) can be written as
\begin{equation}\label{newtaylor}
S_{rs}^{ \rm TN}(A)=\mathbf{1}^TL(A,E_{rs})\mathbf{1}=
\sum_{k=0}^{\infty}\frac{\mathbf{1}^TH_{rs}^{(k)}\mathbf{1}}{(k+1)!}.
\end{equation}

We will first show that all terms in the sum  in \eqref{newtaylor} except for the $2n-1$ first
ones vanish, and then conclude that $S_{rs}^{ \rm TN}(A)$ is maximized for $r=n$ and 
$s=1$. First, note that $\mathbf{1}^TH_{rs}^{(0)}\mathbf{1}=1$, for all $1\leq r,s\leq n$. 
We turn to the expression $\mathbf{1}^TH_{rs}^{(1)}\mathbf{1}$ and use the 
representation $A=\sum_{i=1}^{n-1}\mathbf{e}_i\mathbf{e}_{i+1}^T$, which yields
\[
H_{rs}^{(1)}=AE_{rs}+E_{rs}A=\sum_{i=1}^{n-1}\mathbf{e}_i\mathbf{e}_{i+1}^T
\mathbf{e}_r\mathbf{e}_s^T+\mathbf{e}_r\mathbf{e}_s^T\sum_{i=1}^{n-1}
\mathbf{e}_i\mathbf{e}_{i+1}^T.
\]
Noting that
\begin{eqnarray*}
\sum_{i=1}^{n-1}\mathbf{e}_i\mathbf{e}_{i+1}^T\mathbf{e}_r\mathbf{e}_s^T
&=& \begin{cases}
\mathbf{e}_{r-1}\mathbf{e}_s^T, & \mbox{if } r\geq 2, \\
0,  & \mbox{if } r=1,
\end{cases} \\
\mathbf{e}_r\mathbf{e}_s^T\sum_{i=1}^{n-1}\mathbf{e}_i\mathbf{e}_{i+1}^T
&=& \begin{cases}
\mathbf{e}_r\mathbf{e}_{s+1}^T, & \mbox{if } s\leq n-1, \\
0,  & \mbox{if } s=n,
\end{cases}
\end{eqnarray*}
it follows that $\max_{r,s}\mathbf{1}^TH_{rs}^{(1)}\mathbf{1}=2$ is achieved for all
$2\leq r\leq n$ and $1\leq s\leq n-1$. 

We turn to the expression $H_{rs}^{(2)}=A^2E_{rs}+AE_{rs}A+E_{rs}A^2$ and obtain similarly 
as above that $\max_{r,s}\mathbf{1}^TH_{rs}^{(2)}\mathbf{1}=3$ is achieved for all
$3\leq r\leq n$ and $1\leq s\leq n-2$. Similarly, 
\[
H_{rs}^{(n-1)}=A^{n-1}E_{rs}+A^{n-2}E_{rs}A+A^{n-3}E_{rs}A^2+\dots+E_{rs}A^{n-1}
\]
and $\max_{r,s}\mathbf{1}^TH_{rs}^{(n-1)}\mathbf{1}=n$ for $r=n$ and $s=1$. Our findings 
yield that $\sum_{k=0}^{n-1}\mathbf{1}^TH_{rs}^{(k)}\mathbf{1}$ is maximized for $r=n$ and
$s=1$. Moreover, $\max_{r,s}\mathbf{1}^TH_{rs}^{(k)}\mathbf{1}=k+1$, for $1\leq k\leq n-1$,
i.e., the maximum is the number of terms in the expression for $H_{rs}^{(k)}$. 

Now consider matrices $H_{rs}^{(k)}$ for $k\geq n$. Letting $k=n$ yields
\[
H_{rs}^{(n)}=A^nE_{rs}+A^{n-1}E_{rs}A+A^{n-2}E_{rs}A^2+\dots+E_{rs}A^n,
\]
where we observe that 
\begin{equation}\label{nilpotent}
A^j=0,\qquad j\geq n.
\end{equation}
Thus, the expression for $H_{rs}^{(n)}$ has $n-1$ nonvanishing terms and 
$\max_{r,s}\mathbf{1}^TH_{rs}^{n+1}\mathbf{1}=n-1$ is achieved for $r=n$ and $s=1$. For
the superscript $2n-2$, we have
\[
H_{rs}^{(2n-2)}=A^{2n-2}E_{rs}+A^{2n-3}E_{rs}A+\dots+A^{n-1}E_{rs}A^{n-1}+
\dots+E_{rs}A^{2n-2}.
\]
Due to \eqref{nilpotent}, the only nonvanishing term in the right-hand side is 
$A^{n-1}E_{rs}A^{n-1}$, and we obtain
\[
\max_{r,s}\mathbf{1}^TH_{rs}^{(2n-2)}\mathbf{1}=1.
\]
The maximum is achieved for $r=n$ and $s=1$. We conclude that
$\sum_{k=0}^{2n-2}\mathbf{1}^TH_{rs}^{(k)}\mathbf{1}$ is maximized for $r=n$ and $s=1$. 
Finally, we note that \eqref{nilpotent} implies that $H_{rs}^{(k)}=0$ for $k\geq 2n-1$. It
follows that $\sum_{k=2n-1}^{\infty}\mathbf{1}^TH_{rs}^{(k)}\mathbf{1}=0$.
The above observations show that
\[
\max_{r,s}S_{rs}^{\rm TN}(A)=
\max_{r,s}\sum_{k=0}^{\infty}\frac{\mathbf{1}^TH_{rs}^{(k)}\mathbf{1}}{(k+1)!}=
\max_{r,s}\sum_{k=0}^{2n-2}\frac{\mathbf{1}^TH_{rs}^{(k)}\mathbf{1}}{(k+1)!}=
\sum_{k=0}^{2n-2}\frac{\mathbf{1}^TH_{n1}^{(k)}\mathbf{1}}{(k+1)!}=S_{n1}^{\rm TN}(A),
\]
and the theorem follows.
\end{proof}

\begin{example}\label{ex3.0}
Let the graph $\mathcal{G}$ be defined as in Theorem \ref{thm2} with $n=8$. Its total 
communicability is $11.03$. Consider the the graph $\mathcal{G}'$ obtained by inserting 
the edge $e(v_8\rightarrow v_1)$. It has total communicability $13.75$. If we instead
insert the edge $e(v_1\rightarrow v_3)$ in the graph $\mathcal{G}$, we obtain the 
graph $\mathcal{G}''$ with total communicability $12.75$. The difference in total 
communicability of the graphs $\mathcal{G}'$ and $\mathcal{G}''$ illustrates that the 
choice of edge to insert is important when we aim to increase the total communicability as
much as possible.  $~~\Box$
\end{example}

\begin{example}\label{ex3.1}
Consider the undirected unweighted graph obtained by replacing every directed edge in 
Figure \ref{fig3.0} by an undirected edge and connecting the vertices $v_1$ and $v_8$ by 
the undirected edge $e(v_1\leftrightarrow v_8)$.  We would like to add an undirected edge 
so that the network communicability is increased the most. For notational convenience, we 
identify the node $v_8$ with $v_0$. Due to the circular symmetry of the graph, the total 
network sensitivities are the same in all directions for which $w_{ij}=0$ and $i\neq j$. 
Also, the left and right Perron vectors of the adjacency matrix $A$ are the same. 
Therefore, we cannot determine which edge to add to increase the network communicability 
the most based on these two approaches. However, the Perron network sensitivities 
$S^{\rm PN}_{i,i+4}$ for the most distant nodes (for $i=0,1,\ldots,4$) are the same, and
larger than the sensitivities $S^{\rm PN}_{ij}$ for the nodes $v_i$ and $v_j$ with 
$j\neq i+4$. This suggests that to increase the Perron network communicability the most, 
one should add edges between the most distant nodes. $~~\Box$
\end{example}

\subsection{Two other methods to increase or decrease network communicability}
Arrigo and Benzi \cite{arrigo2016edge} introduced several methods for the selection of 
edges to be added to (or removed from) a given directed or undirected graph defined by the
adjacency matrix $A$ so as to increase or decrease the network communicability. They 
define the \emph{edge total communicability centrality} of an existing edge 
$e(v_i \rightarrow v_j)$ or of a virtual edge $e(v_i \dashrightarrow v_j)$ as
\[
^e TC(i,j)=(e^A\textbf{1})_i(\textbf{1}^Te^A)_j.
\]
They also define another edge total communicability centrality of an existing edge  
$e(v_i\rightarrow v_j)$ or of a virtual edge $e(v_i \dashrightarrow v_j)$ as
\[
^e gTC(i,j)=C_h(i)C_a(j),
\]
where the \emph{total hub communicability} of vertex $v_i$ and the \emph{total authority 
communicability} of vertex $v_j$ are given by   
\[
C_h(i)=[U\sinh({\Sigma})V^T\textbf{1}]_i \quad \text{and} \quad 
C_a(j)=[V\sinh(\Sigma)U^T\textbf{1}]_j,
\]
\noindent 
respectively. Here the matrices $U$, $\Sigma$, and $V$ are the factors of the singular value 
decomposition $A=U\Sigma V^T$.

\begin{figure}[H]
\begin{center}
\begin{tikzpicture}[
 roundnode/.style={circle, draw=black!60, fill=black!10, very thick, minimum size=6mm},
 ->,>=stealth',shorten >=1pt,node distance=2.5cm,auto,thick]

\node[roundnode]      (node1)                                                   {$1$};
\node[roundnode]      (node2)       [right=3cm of node1]      {$2$};
\node[roundnode]      (node3)       [below=3cm of node2]         {$3$};
\node[roundnode]      (node4)       [below=3cm of node1]         {$4$};
\node[roundnode]      (node5)       [below right=0.7cm and 1.4cm of node1]                                           {$5$};
\node[roundnode]      (node6)       [below=0.7cm of node5]      {$6$};
\node[roundnode]      (node7)       [below right=1.5cm and 1cm of node2]         {$7$};

\draw[<->] (node1.0) -- (node2.180);
\draw[->] (node1.-90) -- (node4.90);
\draw[<->] (node2.-145) -- (node5.45);
\draw[->] (node2.-90) -- (node3.90);
\draw[<->] (node2.-45) -- (node7.135);
\draw[->] (node3.-180) -- (node4.0);
\draw[->] (node3.125) -- (node5.-45);
\draw[<->] (node3.45) -- (node7.-115);
\draw[->] (node5.145) -- (node1.-45);
\draw[->] (node5.-90) -- (node6.90);
\draw[->] (node5.0) -- (node7.160);
\draw[->] (node6.45) -- (node2.-125);
\draw[->] (node6.-45) -- (node3.145);
\draw[->] (node6.-135) -- (node4.45);
\draw[->] (node7.-165) -- (node6.0);

\end{tikzpicture}
\caption{Graph of Example 3.5.}
\label{fig6}
\end{center}
\end{figure}
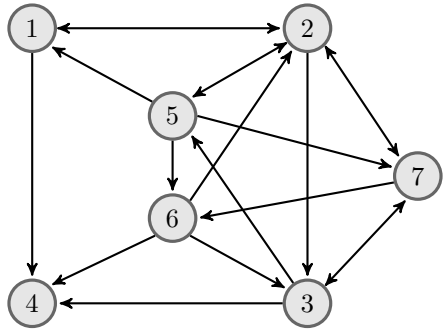

The following example compares the above approaches to the ones of the present paper.

\begin{example}
Regard the directed unweighted graph shown in Figure \ref{fig6}. To obtain an irreducible 
matrix, we use the same procedure as we did in Example \ref{ex_new}, which gives us 
$\delta=10^{-5}$ for this example. We would like to add a 
directed edge so that the network communicability is increased as much as possible. To 
achieve this, the methods by Arrigo and Benzi \cite{arrigo2016edge} described above 
suggest that an edge $e(v_i\rightarrow v_j)$ be inserted into the graph so that the index 
pair $\{i,j\}$ maximizes $^eTC(i,j)$ or $^egTC(i,j)$. For the graph of this example, both 
methods indicate that the edge $e(v_5\rightarrow v_3)$ be added to the graph. Table 
\ref{table10} shows the total network communicability and the Perron 
network communicability after insertion of this edge into the graph. We also evaluate the
Perron root sensitivity ($S^{\rm PR}_{ij}$), the total network sensitivity 
($S^{\rm TN}_{ij}$), and the Perron network sensitivity ($S^{\rm PN}_{ij}$). The total 
network sensitivity is seen to be maximal and the Perron root is increased the most for 
$\{i,j\}=\{7,5\}$, and the Perron network sensitivity is maximized for $\{i,j\}=\{4,5\}$.
Table \ref{table10} shows the addition of the edges $e(v_7\rightarrow v_5)$ or 
$e(v_4\rightarrow v_5)$ to the graph of Figure \ref{fig6} to increase the total network 
communicability and Perron network communicability more than when inserting the edge 
$e(v_5\rightarrow v_3)$. We remark that the selection criteria used in the methods 
\cite{arrigo2016edge} perform well for many graphs, but not for all.~~$\Box$

\begin{table}[H]
\renewcommand{\arraystretch}{1.15}
\begin{center}
\begin{tabular}{cccc}
\hline  
\textbf{Methods} & $\{i,j\}$ & $C^{\rm TN}(A+E_{ij})$ &   $C^{\rm PN}(\hat{A}+E_{ij})$  \\
\hline
 $^eTC(i,j)$                                   &  $\{5, 3\}$    &  117.3601    &    92.4046      \\
 $^egTC(i,j)$                                 &  $\{5, 3\}$    &  117.3601    &    92.4046     \\
$S^{\rm TN}_{ij}$                          &  $\{7, 5\}$    &  127.1123    &    92.4049     \\
$S^{\rm PN}_{ij}$                          &  $\{4, 5\}$    &  124.1918    &    92.4050      \\
$S^{\rm PR}_{ij}$                          &  $\{7, 5\}$    &  127.1123    &    92.4049     \\

\hline
\end{tabular}\\
\vspace{0.05in}
\end{center} 
\caption{Example \ref{ex7}: The second column lists the edge to be added, and the third 
and fourth columns show the total network communicability and the Perron network 
communicability, respectively, when $w_{ij}$ is increased from 0 to 1.}
\label{table10}
\end{table}
\label{ex7}
\end{example}

\section{Efficient methods for large-scale networks }\label{lsn}
This section discusses some numerical methods for estimating the total network 
sensitivity, the Perron network sensitivity, and the Perron root sensitivity for 
large-scale networks. Subsections \ref{arnoldimtd}, \ref{lanczosmtd}, and \ref{KKRS} 
describe five iterative Krylov subspace methods to estimate the total network 
sensitivity. Algorithms for estimating the Perron network sensitivity and the Perron root
sensitivity are considered in Subsection \ref{ptwomtd}.

\subsection{Applications of the Arnoldi process to large-scale network problems}
\label{arnoldimtd}
The evaluation of the total network sensitivity involves the computation of Fr\'{e}chet 
derivatives, which can be done, e.g., by using (\ref{thm1result}). However, this approach 
is quite expensive when the adjacency matrix $A$ is large. This section, therefore, 
describes several iterative Krylov subspace methods to estimate the total network 
sensitivity. These methods are much cheaper than straightforward evaluation of 
(\ref{thm1result}) when the adjacency matrix is large and sparse. Application of 
$1\leq m\ll 2n$ steps of the Arnoldi process to the matrix
\begin{equation}\label{eqmm}
M=\begin{bmatrix}
A & E_{ij} \\
0 & A \\
\end{bmatrix}\in\R^{2n\times 2n}
\end{equation}
with initial unit vector $\mathbf{v}_1=n^{-1/2}[\mathbf{0}^T~ \mathbf{1}^T]^T\in\R^{2n}$
gives the Arnoldi decomposition 
\begin{equation}\label{arnoldi}
MV_m=V_mH_m+\mathbf{g}\mathbf{e}_m^T,
\end{equation}
under the assumption that no breakdown occurs. Here, $\mathbf{0}=[0,\ldots,0]^T\in\R^n$,
$\mathbf{1}=[1,\ldots,1]^T\in\R^n$, and 
$E_{ij}=\mathbf{e}_i\mathbf{e}_j^T\in\R^{n\times n}$. The matrix 
$H_m\in\R^{m\times m}$ is of upper Hessenberg form with nonvanishing subdiagonal entries.
The columns of the matrix 
$V_m=[\mathbf{v}_1,\mathbf{v}_2,\ldots,\mathbf{v}_m]\in\R^{2n\times m}$ form an 
orthonormal basis for the Krylov subspace
\[
\mathcal{K}_m(M,\mathbf{v}_1):=
{\rm span}\{\mathbf{v}_1,M\mathbf{v}_1,\ldots,M^{m-1}\mathbf{v}_1\},
\]
and $\mathbf{g}\in\R^{2n}$ is orthogonal to $\mathcal{K}_m(M,\mathbf{v}_1)$; see, e.g., 
Saad \cite{saad2003iterative} for details on the Arnoldi process. Breakdown of the 
Arnoldi process occurs when a subdiagonal entry of $H_m$ vanishes. We will not dwell on 
this rare situation. It is well known that
\begin{equation}\label{pMv1}
p(M)\mathbf{v}_1=V_m p(H_m)\mathbf{e}_1,
\end{equation}
for all polynomials $p$ of degree at most $m-1$; see, e.g., \cite{Sa2}. 

We apply the decomposition \eqref{arnoldi} to compute an approximation of
$\mathbf{1}^TL(A,E_{ij})\mathbf{1}$ using \eqref{thm1result} as follows: Define the unit 
vector $\mathbf{w}=n^{-1/2}[\mathbf{1}^T~ \mathbf{0}^T]^T\in\R^{2n}$. Then
\begin{equation}\label{approx1}
\mathbf{1}^TL(A,E_{ij})\mathbf{1}=n\mathbf{w}^T\exp_0(M)\mathbf{v}_1\approx
n\mathbf{w}^TV_m\exp_0(H_m)\mathbf{e}_1.
\end{equation}
Bounds for the discrepancy $\exp_0(M)\mathbf{v}_1-V_m\exp_0(H_m)\mathbf{e}_1$ can be found
in, e.g., \cite{BR}. The columns of $V_m$ should be reorthogonalized when computed by the
Arnoldi process to secure that the vector $\mathbf{w}^TV_m$ can be evaluated accurately;
see \cite{saad2003iterative} for a discussion and implementation of the Arnoldi process
with reorthogonalization.

Formula \eqref{defL} provides an alternative approach to computing an approximation of 
$\mathbf{1}^TL(A,E_{ij})\mathbf{1}$ by the Arnoldi process. It follows from \eqref{defL} 
that
\[
\lim_{t\rightarrow 0}\frac{\exp_0(A+t E_{ij}) - \exp_0(A)}{t} = L(A,E_{ij}),
\]
which suggests that we apply the Arnoldi process to the matrices $A$ and $A+tE_{ij}$ for 
some small $t>0$ separately with initial unit vector 
$\widehat{\mathbf{v}}_1=n^{-1/2}\mathbf{1}\in\R^n$. Thus, application of $m$ steps of the 
Arnoldi process to $A$ and $A+tE_{ij}$ with initial vector $\widehat{\mathbf{v}}_1$ yields 
the Arnoldi decompositions
\begin{equation}\label{arndec2}
A\widehat{V}_m=\widehat{V}_m\widehat{H}_m+\widehat{\mathbf{g}}\mathbf{e}_m^T,\qquad
(A+tE_{ij})\widetilde{V}_m=\widetilde{V}_m\widetilde{H}_m+
\widetilde{\mathbf{g}}\mathbf{e}_m^T,
\end{equation}
where the columns of $\widehat{V}_m$ and $\widetilde{V}_m$ form orthonormal bases for the
Krylov subspaces $\mathcal{K}_m(A,\widehat{\mathbf{v}}_1)$ and
$\mathcal{K}_m(A+tE_{ij},\widehat{\mathbf{v}}_1)$, respectively, the matrices 
$\widehat{H}_m,\widetilde{H}_m\in\R^{m\times m}$ are of upper Hessenberg form, and the 
$n$-vectors $\widehat{\mathbf{g}}$ and $\widetilde{\mathbf{g}}$ are orthogonal to the 
Krylov subspaces $\mathcal{K}_m(A,\widehat{\mathbf{v}}_1)$ and
$\mathcal{K}_m(A+tE_{ij},\widehat{\mathbf{v}}_1)$, respectively. For some small $t>0$, we
use the approximations
\[
L(A,E_{ij})\mathbf{1}\approx \frac{\exp_0(A+tE_{ij})-\exp_0(A)}{t} \mathbf{1}\approx
n^{1/2}\frac{\widetilde{V}_m\exp_0(\widetilde{H}_m)\mathbf{e}_1-
\widehat{V}_m\exp_0(\widehat{H}_m)\mathbf{e}_1}{t},
\]
which yield
\begin{equation}\label{approx2}
\mathbf{1}^TL(A,E_{ij})\mathbf{1}\approx 
n\frac{\mathbf{e}_1^T\exp_0(\widetilde{H}_m)\mathbf{e}_1-
\mathbf{e}_1^T\exp_0(\widehat{H}_m)\mathbf{e}_1}{t}.
\end{equation}
We will illustrate the use of the right-hand sides \eqref{approx1} and \eqref{approx2} in 
computed examples in Section \ref{largeexamples}.

\subsection{Applications of the Lanczos biorthogonalization algorithm to large-scale 
network problems}\label{lanczosmtd}
We describe how the expression (\ref{thm1result}) can be approximated by carrying out a
few steps of the Lanczos biorthogonalization algorithm \cite{saad2003iterative}. This
approach is an alternative to the application of the Arnoldi algorithm described above.
Let the vector $\mathbf{v}_1$ and matrix $E_{ij}$ be the same as in Subsection 
\ref{arnoldimtd} and define $\mathbf{w}_1=\mathbf{v}_1$. Application of $1\leq m \ll 2n$ 
steps of the Lanczos biorthogonalization algorithm to the matrix \eqref{eqmm} with unit
starting vectors $\mathbf{v}_1$ and $\mathbf{w}_1$ gives, in the absence of breakdown of 
the recursion formulas, the decompositions 
\[
\begin{array}{rcl}
MV_m&=&V_mT_m+\mathbf{g}_1\mathbf{e}_m^T, \\
M^TW_m&=&W_mT_m^T+\mathbf{g}_2\mathbf{e}_m^T,
\end{array}
\]
where the matrix $T_m\in\R^{m\times m}$ is tridiagonal and the columns of the matrices 
$V_m=[\mathbf{v}_1,\mathbf{v}_2,\ldots,\mathbf{v}_m]\in\R^{2n\times m}$ and 
$W_m=[\mathbf{w}_1,\mathbf{w}_2,\ldots,\mathbf{w}_m]\in\R^{2n\times m}$ form a pair of 
biorthogonal bases for the Krylov subspaces
\begin{eqnarray*}
\mathcal{K}_m(M,\mathbf{v}_1)&:=&
{\rm span}\{\mathbf{v}_1,M\mathbf{v}_1,\ldots,M^{m-1}\mathbf{v}_1\},\\
\mathcal{K}_m(M^T,\mathbf{w}_1)&:=&
{\rm span}\{\mathbf{w}_1,M^T\mathbf{w}_1,\ldots,(M^T)^{m-1}\mathbf{w}_1\},
\end{eqnarray*}
respectively. The vectors $\mathbf{g}_1,\mathbf{g}_2\in\R^{2n}$ satisfy certain 
orthogonality relations; see, e.g., Saad \cite{saad2003iterative} for details. Moreover, we 
have analogously to \eqref{pMv1} that
\[
p(M)\mathbf{v}_1=V_mp(T_m)\mathbf{e}_1
\]
for all polynomials $p$ of degree at most $m-1$.

Let $\mathbf{w}=n^{-1/2}[\mathbf{1}^T~\mathbf{0}^T]^T\in\R^{2n}$. Then an approximation 
of $\mathbf{1}^TL(A,E_{ij})\mathbf{1}$ analogous to the one determined by application of
$m$ steps of the Arnoldi process is given by
\begin{equation}\label{lbapprox1}
\mathbf{1}^TL(A,E_{ij})\mathbf{1}=n\mathbf{w}^T\exp_0(M)\mathbf{v}_1\approx
n\mathbf{w}^TV_m\exp_0(T_m)\mathbf{e}_1.
\end{equation}
To assure that the vector $\mathbf{w}^TV_m$ can be calculated accurately, the columns of 
the matrices $V_m$ and $W_m$ should be rebiorthogonalized when computed by the Lanczos 
biorthogonalization algorithm; see Parlett et al. \cite{PTL} for a discussion.

We also will illustrate the following alternative way of using the Lanczos 
biorthogonalization algorithm to approximate $\mathbf{1}^TL(A,E_{ij})\mathbf{1}$. 
Application of $m$ steps of this algorithm to the matrices $A$ and $A+tE_{ij}$ for some 
small $t>0$ separately with initial unit vectors 
$\widehat{\mathbf{v}}_1=\widehat{\mathbf{w}}_1=n^{-1/2}\mathbf{1}\in\R^n$ yields, assuming
that no breakdown occurs, the decompositions 
\begin{equation}\label{lbapprox2}
\begin{cases}
A\widehat{V}_m=\widehat{V}_m\widehat{T}_m+\widehat{\mathbf{g}}_1\mathbf{e}_m^T,\\
A^T\widehat{W}_m=\widehat{W}_m\widehat{T}_m^T+\widehat{\mathbf{g}}_2\mathbf{e}_m^T,
\end{cases}
\quad \text{and~~} \quad
\begin{cases}
(A+tE_{ij})\widetilde{V}_m=\widetilde{V}_m\widetilde{T}_m+
\widetilde{\mathbf{g}}_1\mathbf{e}_m^T,\\
(A+tE_{ij})^T\widetilde{W}_m=\widetilde{W}_m\widetilde{T}_m^T+
\widetilde{\mathbf{g}}_2\mathbf{e}_m^T,
\end{cases}
\end{equation}
where the columns of the matrices $\widehat{V}_m$ and $\widehat{W}_m$ form a pair of 
biorthogonal bases for the Krylov subspaces 
\begin{eqnarray*}
\mathcal{K}_m(A,\widehat{\mathbf{v}}_1)&:=&
{\rm span}\{\widehat{\mathbf{v}}_1,A\widehat{\mathbf{v}}_1,\ldots, A^{m-1}
\widehat{\mathbf{v}}_1\},\\
\mathcal{K}_m(A^T,\widehat{\mathbf{w}}_1)&:=&
{\rm span}\{\widehat{\mathbf{w}}_1,A^T\widehat{\mathbf{w}}_1,\ldots, (A^T)^{m-1}
\widehat{\mathbf{w}}_1\},
\end{eqnarray*}
respectively, the columns of the matrices $\widetilde{V}_m$ and $\widetilde{W}_m$ form a 
pair of biorthogonal bases for the Krylov subspaces
\begin{eqnarray*}
\mathcal{K}_m(A+tE_{ij},\widehat{\mathbf{v}}_1)&:=&
{\rm span}\{\widehat{\mathbf{v}}_1,(A+tE_{ij})\widehat{\mathbf{v}}_1,\ldots, 
(A+tE_{ij})^{m-1}\widehat{\mathbf{v}}_1\},\\
\mathcal{K}_m((A+tE_{ij})^T,\widehat{\mathbf{w}}_1)&:=&
{\rm span}\{\widehat{\mathbf{w}}_1,(A+tE_{ij})^T\widehat{\mathbf{w}}_1,\ldots, 
((A+tE_{ij})^T)^{m-1}\widehat{\mathbf{w}}_1\},
\end{eqnarray*}
respectively, the matrices $\widehat{T}_m,\widetilde{T}_m\in\R^{m\times m}$ are 
tridiagonal, and the vectors $\widehat{\mathbf{g}}_1,\widehat{\mathbf{g}}_2\in\R^n$ 
satisfy certain orthogonality conditions. 

We apply the decompositions \eqref{lbapprox2} similarly as we used the decompositions
\eqref{arndec2}. Thus, for some small $t>0$, we use the approximations
\[
L(A,E_{ij})\mathbf{1}\approx \frac{\exp_0(A+tE_{ij})-\exp_0(A)}{t} \mathbf{1}\approx
n^{1/2}\frac{\widetilde{V}_m\exp_0(\widetilde{T}_m)\mathbf{e}_1-
\widehat{V}_m\exp_0(\widehat{T}_m)\mathbf{e}_1}{t},
\]
which yield
\begin{equation}\label{lbapprox3}
\mathbf{1}^TL(A,E_{ij})\mathbf{1}\approx 
n\frac{\mathbf{e}_1^T\exp_0(\widetilde{T}_m)\mathbf{e}_1-
\mathbf{e}_1^T\exp_0(\widehat{T}_m)\mathbf{e}_1}{t}.
\end{equation}
We will illustrate the use of the right-hand sides \eqref{lbapprox1} and \eqref{lbapprox3} 
in computed examples in Section \ref{largeexamples}.

\subsection{Another Arnoldi-based method for approximating the Fr\'{e}chet derivative}
\label{KKRS}
Kandolf et al. \cite{kkrs2020computing} introduced several Krylov subspace methods for 
approximating the Fr\'{e}chet derivative. They are defined with the aid of Cauchy
integrals and are based on the Lanczos, Arnoldi, and two-sided Arnoldi processes. We
outline the Arnoldi-based method; its performance will be illustrated in Section 
\ref{largeexamples}.

Consider a directed graph with $n$ nodes and define the associated non-symmetric 
adjacency matrix $A\in\R^{n\times n}$ and the direction matrix 
$E=\eta\mathbf{y}\mathbf{z}^T\in\R^{n\times n}$ of rank one. Here $\eta\in\R$ and 
$\mathbf{y},\mathbf{z}\in\R^n$ are unit vectors. Kandolf et al. \cite{kkrs2020computing} 
described the following approach to approximate the Fr\'{e}chet derivative of $\exp_0(A)$ 
with respect to the direction $E$. 

Application of $1\leq m\ll n$ steps of the Arnoldi process to the matrices $A$ and $A^T$
with initial vectors $\mathbf{y}$ and $\mathbf{z}$, respectively, gives, in the absence of
breakdown of the recursion formulas, the Arnoldi decompositions
\begin{eqnarray*}
AV_m&=&V_mG_m+\mathbf{g}_1\mathbf{e}_m^T,\\ 
A^TW_m&=&W_mH_m+\mathbf{g}_2\mathbf{e}_m^T,
\end{eqnarray*}
where the matrices $G_m,H_m\in\R^{m\times m}$ are of upper Hessenberg form with 
nonvanishing subdiagonal entries. The columns of the matrices 
$V_m=[\mathbf{v}_1,\mathbf{v}_2,\ldots,\mathbf{v}_m]\in\R^{n\times m}$ and 
$W_m=[\mathbf{w}_1,\mathbf{w}_2,\ldots,\mathbf{w}_m]\in\R^{n\times m}$ form orthonormal 
bases for the Krylov subspaces
\begin{eqnarray*}
\mathcal{K}_m(A,\mathbf{y})&:=&{\rm span}\{\mathbf{y},A\mathbf{y},\ldots,
A^{m-1}\mathbf{y}\},\\
\mathcal{K}_m(A^T,\mathbf{z})&:=&{\rm span}\{\mathbf{z},A^T\mathbf{z},\ldots,
(A^T)^{m-1}\mathbf{z}\},
\end{eqnarray*}
respectively, with $\mathbf{v}_1=\mathbf{y}$ and $\mathbf{w}_1=\mathbf{z}$. The vectors 
$\mathbf{g}_1,\mathbf{g}_2\in\R^n$ are orthogonal to 
$\mathcal{K}_m(A,\mathbf{y})$ and $\mathcal{K}_m(A^T,\mathbf{z})$, respectively. Let
\[
B = 
\begin{bmatrix}
G_m & \eta\,\mathbf{e}_1\mathbf{e}_1^T \\
0 & H_m^T \\
\end{bmatrix}.
\]
Then the $m$th Arnoldi approximation of the Fr\'{e}chet derivative is given by
\[
L^{\rm Arn}_m:= \eta V_mX_mW_m^T,
\]
where $X_m$ can be computed using the equation
\[
\exp_0(B)=
\begin{bmatrix}
\exp_0(G_m) & X_m \\
0 & \exp_0(H_m^T) \\
\end{bmatrix}, 
\]
and the $m$th Arnoldi approximation for the total network sensitivity is
\begin{equation}\label{kkrseqn}
\mathbf{1}^TL^{\rm Arn}_m\mathbf{1} = \eta\mathbf{1}^T V_mX_mW_m^T\mathbf{1};
\end{equation}
see, e.g., Kandolf et al. \cite{kkrs2020computing} for details. We refer to this method as
the KKRS Arnoldi method in Section \ref{largeexamples}.

\subsection{Applications of the two-sided Arnoldi and restarted Lanczos methods to 
large-scale network problems}\label{ptwomtd}
The dominant computational burden when studying the Perron root sensitivity and 
evaluating the quantities $C^{\rm PN}(A+E_{ij})$ and $S^{\rm PN}_{ij}$ is the calculation of the 
Perron root and the left and right Perron vectors. For small networks, these
quantities easily can be evaluated by using MATLAB functions \textbf{eig} or \textbf{eigs}.
For large-scale networks, when the graph ${\mathcal G}$ that determines $A$ is directed, 
and $A$ therefore is nonsymmetric, these quantities typically can be computed fairly 
inexpensively by the two-sided Arnoldi method, which was introduced by Ruhe \cite{Ru}, and 
recently has been investigated and improved by Zwaan and Hochstenbach \cite{ZH}. In the 
situation when $A$ is symmetric, a restarted Lanczos method, such as \cite{BCR}, can be 
applied.

\section{Examples with Large-Scale Networks}\label{largeexamples}
This section presents examples with large-scale networks to illustrate the performance of
the numerical methods described in Section \ref{lsn}. The computations are carried out 
using MATLAB R2018b on an Intel Xeon Silver 4116 CPU @ 2.10 GHz
(48 cores, 96 threads) equipped with 256 Gbyte RAM. The USAir97 data set 
used in Example \ref{ex8} can be downloaded from the website \cite{air97&usroads48}, 
the Air500 data set used in Example \ref{ex9} can be downloaded from \cite{air500},
and the usroads-48 data set used in Example \ref{ex10} can be downloaded from 
\cite{air97&usroads48}.

Define for notational convenience the relative difference
\begin{equation}\label{reldiff}
r_{ij}:=|(\text{new approximation})-(\text{previous approximation})|/
|(\text{previous approximation})|,
\end{equation}
where ``previous approximation'' and ``new approximation''  denote approximations of the 
total network sensitivity $S^{\rm TN}_{ij}$ with respect to the direction
$E_{ij}=\mathbf{e}_i\mathbf{e}_j^T$ determined by carrying out $m$ and $m+1$ steps,
respectively, of an iterative method.

When applying the methods of Section \ref{lsn} to estimate the total network sensitivity, 
we terminate the iterations as soon as $r_{ij}<10^{-4}$ for each $S^{\rm TN}_{ij}$. We 
refer to the exact total network sensitivities as the ``exact solution'', and denote the 
approximate solutions obtained by using the right-hand sides of eqs. \eqref{approx1}, 
\eqref{approx2}, and \eqref{kkrseqn} by the ``Arnoldi solution \eqref{approx1}'', the 
``Arnoldi solution \eqref{approx2}'', or the ``KKRS Arnoldi solution'', respectively. 
Similarly, the approximate solutions determined by the approximations \eqref{lbapprox1} 
and \eqref{lbapprox3} are referred to as the ``Lanczos solution \eqref{lbapprox1}'' and 
the ``Lanczos solution \eqref{lbapprox3}'', respectively. We let $\eta=1$ in the KKRS 
Arnoldi method, and $t=2\cdot10^{-5}$ in the methods \eqref{approx2} and 
\eqref{lbapprox3}, as well as in the computation of 
$S^{\rm PN}_{ij}=\mathbf{1}^TL^{\rm PN}(A,E_{ij})\mathbf{1}$. We use the two-sided Arnoldi
method to compute the Perron root, and left and right Perron vectors.

\begin{example}
We consider the network USAir97, which is represented by an undirected weighted graph.
The graph has $332$ nodes, which correspond to American airports in 1997. Undirected 
edges represent flights from one airport to another and the weight of each undirected edge 
indicates the frequency of flights between airports. The adjacency matrix $A$ for the 
graph is irreducible. Our aim is to determine an edge $e(v_i\leftrightarrow v_j)$ 
(or $e(v_i~\LeftArrow{0.25cm}\RightArrow{0.34cm}~v_j)\notin\mathcal{E}$) for $i\neq j$ such
that the network communicability is increased the most when increasing the edge-weight
$w_{ij}$ and $w_{ji}$ slightly. Thus, to preserve symmetry our perturbations are 
multiples of $E_{ij}+E_{ji}$ for different $i$ and $j$.

Since the adjacency matrix $A$ is symmetric, we only compute the Perron root sensitivities 
$S^{\rm PR}_{ij}$, and explore the total network sensitivity and the Perron network 
sensitivity to changes in the weights $w_{ij}$ and $w_{ji}$ of edges $e(v_i\leftrightarrow v_j)$ 
(or $e(v_i~\LeftArrow{0.25cm}\RightArrow{0.34cm}~v_j)\notin\mathcal{E}$) with 
$j>i$ for $i,j\in\{1,2,\ldots,n\}$. Each edge-weight is increased in turn and the quantities 
$S^{\rm TN}_{ij}$ and $S^{\rm PN}_{ij}$ are recorded to find the edge, whose associated 
weight should be increased. We also evaluate the total network communicability and the 
Perron network communicability of the graphs obtained when increasing each pair of weights 
$w_{ij}$ and $w_{ji}$ by one. Table \ref{tableev1} displays the five largest Perron root 
sensitivities $S^{\rm PR}_{ij}$ and shows that the Perron root is increased the most when 
increasing the weight of edge $e(v_{248}\leftrightarrow v_{201})$ slightly. The five 
largest total network sensitivities $S^{\rm TN}_{ij}$ and Perron network sensitivities 
$S^{\rm PN}_{ij}$ with respect to the weight of edge $e(v_i\leftrightarrow v_j)$, together with the total network 
communicability and Perron network communicability attained as the weight of edge $e(v_i\leftrightarrow v_j)$ is 
increased by one are reported in Table \ref{table11}. The table shows both the total network 
communicability and the Perron network communicability to increase the most when the weight 
of the edge $e(v_{118}\leftrightarrow v_{248})$ is increased by one. 
Since $w_{118,248}=w_{248,118}=0.1733$ this means one should increase the frequency
of flights between airport $118$ and airport $248$ to increase the network communicability 
with respect to any one edge-weight the most. In this example, since the largest Perron root
sensitivities $S^{\rm PR}(A)_{i,j}$ are very close in size, the edge selection furnished by the
Perron root sensitivity does not give the largest increase in the Perron network communicability, 
or the total network communicability.

Table \ref{table12} shows the CPU time required for computing the Perron root sensitivity, 
the Perron network sensitivity, and the total network sensitivity. For the latter, 
the CPU time for the exact solution and five approximate solutions determined by the 
Arnoldi  and Lanczos biorthogonalization methods is reported. The table also  displays
the average number of steps needed by the Arnoldi and Lanczos 
biorthogonalization methods to satisfy the stopping criterion. Both the Perron network 
sensitivity and the Perron root sensitivity suggest that the same edge-weight be increased, 
but the computation of the Perron root sensitivity is much cheaper than evaluating the 
Perron network sensitivity for large networks, as shown in Table \ref{table12}. Indeed, the
Perron root sensitivity may be the only practical indicator of which edge-weight to
modify for very large networks. The ``exact solution'' in the table is evaluated by
computing $L(A,E_{ij})$ using \eqref{thm1result} with $f$ the matrix exponential.

We implemented the Arnoldi method \eqref{approx1} with reorthogonalization to avoid that 
the computed results are influenced by loss of orthogonality of the Krylov subspace basis.
Similarly, rebiorthogonalization is performed for the Lanczos biorthogonalization method 
\eqref{lbapprox1}. 

Reducing the tolerance for $r_{ij}$ in \eqref{reldiff} to $10^{-5}$ to decide when to 
terminate the iterations did not change the edges selected when using the Arnoldi-based or
Lanczos-based methods. 

Figure \ref{fig7} shows the exact total network sensitivities and 
approximate total network sensitivities determined by the five iterative Krylov subspace 
methods described in Section \ref{lsn}. Specifically, assume that we are interested in evaluating 
the total network sensitivity $S^{\rm TN}_{ij}$ for $i\ne j$ and $j>i$ in $\ell$ 
different directions. We consider increasing one edge-weight of $A$ at each step, then calculate its 
associated exact (or approximate)
total network sensitivity $S^{\rm TN}_{ij}$, denote it by $s_1^{exact}$ (or $s_1^{approx}$), 
and store it in the vectors $\mathbf{s}^{exact}=[s_1^{exact}]$ (or 
$\mathbf{s}^{approx}=[s_1^{approx}]$). 
Repeating this procedure for the $\ell$ directions, we get the vectors
$\mathbf{s}^{exact}=[s_1^{exact},s_2^{exact},\ldots,s_\ell^{exact}]^T\in\R^{\ell}$ 
and $\mathbf{s}^{approx}=[s_1^{approx},s_2^{approx},\ldots,s_\ell^{approx}]^T\in\R^{\ell}$.
Each subfigure of Figure \ref{fig7} displays the vectors $\mathbf{s}^{exact}$ on the vertical 
axis and $\mathbf{s}^{approx}$ for one of the Krylov methods on the horizontal axis. Each 
pair of vector entries $\{s_i^{approx},s_i^{exact}\}$ gives one dot on the graph; if the 
approximations are accurate, then $s_i^{approx}$ is close to $s_i^{exact}$ and the resulting 
graph is on a straight line. We conclude that the Arnoldi method \eqref{approx2}
is the fastest but the KKRS Arnoldi method is most reliable. The Arnoldi methods demand $\overline{m}=9.4774$,
$\overline{m}=4.9942$, or $\overline{m}=6.7381$ average number of steps to satisfy the stopping criterion, 
while the Lanczos biorthogonalization methods require $10.9233$ or $4.9941$ average number 
of steps. The number of matrix-vector product evaluations needed by the Arnoldi methods \eqref{approx1} 
and \eqref{approx2} are smaller, since each step only requires one matrix-vector product evaluation, while each 
step of the KKRS Arnoldi and Lanczos biorthogonalization methods requires two matrix-vector 
product evaluations, one with $A$ (or $M$), and one with $A^T$ (or $M^T$). Each 
matrix-vector product evaluation with $M$ requires two matrix-vector product evaluations 
with $A$. We conclude from Figure \ref{fig7} that the application of the Arnoldi method 
based on \eqref{approx2} requires the fewest matrix-vector product evaluations followed by
the Lanczos biorthogonalization method based on \eqref{lbapprox3}.$~~\Box$

\begin{table}[H]
\renewcommand{\arraystretch}{1.15}
\begin{center}
\begin{tabular}{cc}
\hline  
$\{i,j\}$ & $S^{\rm PR}_{ij}$ \\
\hline
$\{248,201\}$     &  0.0471    \\
$\{248,47\}$       &  0.0470      \\
$\{201,47\}$       &  0.0464      \\
$\{248,118\}$     &  0.0452       \\
$\{118,201\}$     &  0.0446       \\
\hline
\end{tabular}\\
\vspace{0.05in}
\end{center} 
\caption{Example \ref{ex8}: The five largest Perron root sensitivities along directions for which $i\neq j$.}
\label{tableev1}
\end{table}

\begin{table}[H]
\renewcommand{\arraystretch}{1.15}
\begin{center}
\begin{tabular}{cccccc}
\hline  
$\{i,j\}$ & $S^{\rm TN}_{ij}$ & $C^{\rm TN}(A+E_{ij}+E_{ji})$ & $\{i,j\}$ & $S^{\rm PN}_{ij}$ & 
$C^{\rm PN}(A+E_{ij}+E_{ji})$ \\
\hline
$\{118,248\}$   & 408.5090   & 5605.03    & $\{118,248\}$    &   412.30   &  5600.87  \\
$\{47,118\}$     & 402.6977   & 5597.83    & $\{47,118\}$      &   406.85   &  5594.19  \\
$\{118,201\}$   & 401.8136   & 5596.48   & $\{118,201\}$     &   405.79   &  5592.67 \\
$\{118,261\}$   & 396.2688   & 5590.28   & $\{118,261\}$     &   401.14   &  5590.54 \\
$\{67,118\}$     & 386.3636   & 5576.64   & $\{67,118\}$       &   390.32   &  5573.10\\
\hline
\end{tabular}\\
\vspace{0.05in}
\end{center} 
\caption{Example \ref{ex8}: The five largest total network sensitivities $S^{\rm TN}_{ij}$
and Perron network sensitivities $S^{\rm PN}_{ij}$ with respect to changes in the weights 
of the edge $e(v_i \leftrightarrow v_j)$ with $i\neq j$ and $j>i$, as well 
as the corresponding total network communicability and Perron network communicability 
when the weight of the edge $e(v_i\leftrightarrow v_j)$ is increased by one.}\label{table11}
\end{table}

\begin{table}[H]
\renewcommand{\arraystretch}{1.15}
\begin{center}
\begin{tabular}{lcccc}
\hline  
\textbf{Communicability} & \textbf{Method} & \textbf{Average steps}& \textbf{Elapsed time (in seconds)} \\
\hline  
Perron network communicability   &   Perron root sensitivity                         &  N/A           &  0.1                        \\
                                                      &   Perron network sensitivity                  &  N/A            & 3180 (53 mins)       \\
\hline
                                                       &    Exact solution                                   &  N/A            &  56697 (15.7 hrs)     \\
                                                       &    Arnoldi solution \eqref{approx1}        & 9.4774       &  253                     \\
Total network communicability       &    Arnoldi solution \eqref{approx2}        & 4.9942       &  105                       \\
                                                       &    Lanczos solution \eqref{lbapprox1}   & 10.9233     &  1032                       \\
                                                       &    Lanczos solution \eqref{lbapprox3}   & 4.9941       &  113                       \\
                                                       &    KKRS Arnoldi solution                       & 6.7381       &  226                        \\
\hline
\end{tabular}\\
\vspace{0.05in}
\end{center} 
\caption{CPU time required for evaluating the Perron root sensitivity, the Perron 
network sensitivity, the exact solution, and the five approximate solutions, as well as 
the average number of steps $\overline{m}$ demanded by each iterative method to satisfy the stopping 
criterion.}\label{table12}
\end{table}

\begin{figure}[H]
\begin{center}
\includegraphics[scale=0.42]{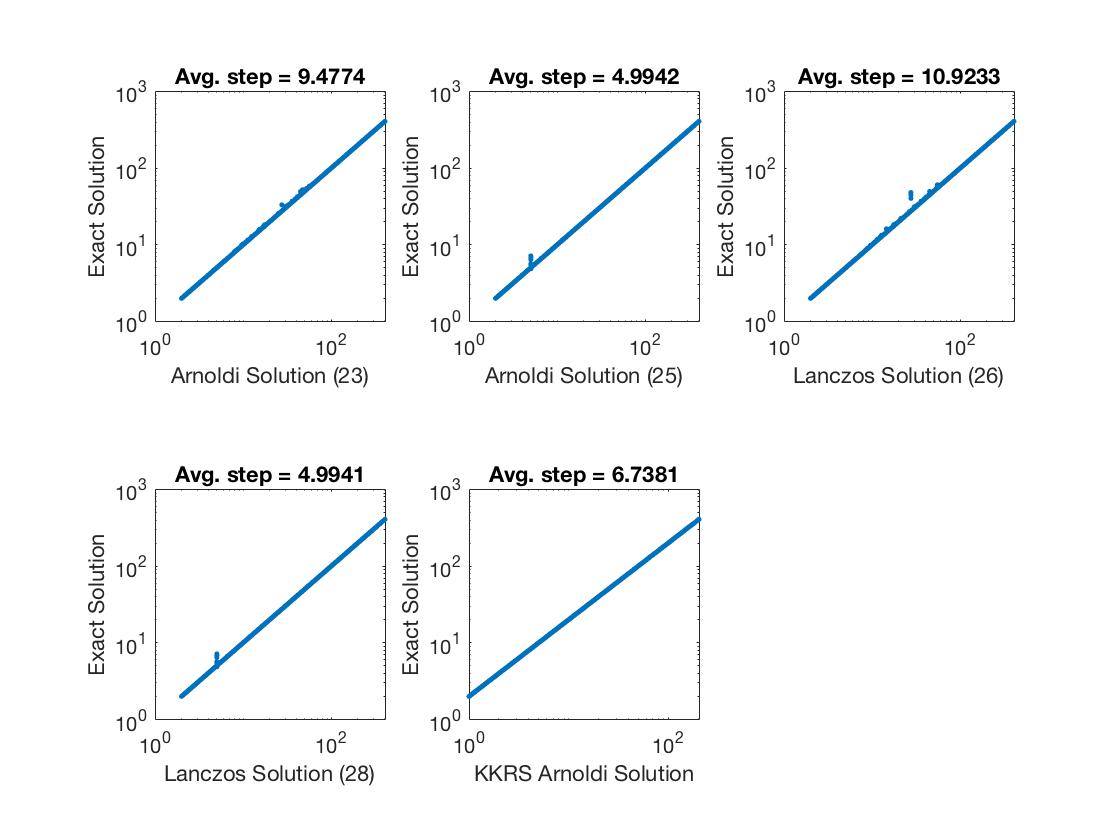}
\caption{USAir97: Comparison of the exact and approximate solutions using the five 
iterative Krylov subspace methods described in Section \ref{lsn}. The heading of each 
subplot shows the average number of steps $\overline{m}$ required by each iterative method to satisfy the stopping 
criterion.}\label{fig7}
\end{center}
\end{figure}

\label{ex8}
\end{example}
 
\begin{example} 
The network Air500 is represented by an unweighted directed graph with $500$ nodes, which model the
top 500 airports worldwide based on total passenger volume. Flights are represented by
edges; the graph is based on flights within one year from July 1, 2007, to June 30, 2008. The
graph is strongly connected. The adjacency matrix associated with the graph therefore is
irreducible.

In this example, we want to insert an edge $e(v_i\dashrightarrow v_j)\notin\mathcal{E}$ 
that enhances the network communicability the most. We study the Perron root sensitivity  
by computing the Perron root sensitivity matrix of the adjacency matrix $A$, and 
investigate the total network sensitivity and the Perron network sensitivity with respect 
to weights $w_{ij}$, where $w_{ij}=0$ and $i\neq j$. Each edge is added in turn, and the 
quantities $S^{\rm TN}_{ij}$ and $S^{\rm PN}_{ij}$ are recorded to find the optimal edge. 
We also evaluate the total network communicability and the Perron network communicability 
of the graph attained when $w_{ij}$ is changed from $0$ to $1$. 

We recall that the Perron root is increased the most when the Perron root sensitivity 
$S^{\rm PR}_{ij}$ is maximized. In this example, we are only interested in adding an edge 
when $w_{ij}=0$ and $i\neq j$. The five largest Perron root sensitivities 
$S^{\rm PR}_{ij}$ of interest are listed in Table \ref{tableev2}. The table indicates that
the Perron root is increased the most by inserting the edge 
$e(v_{224}\rightarrow v_{257})$ into the graph. The network communicability also is 
increased the most by adding this edge as is seen in Table \ref{table13}. This table 
displays the five largest total network sensitivities and Perron network sensitivities
with respect to changes in $w_{ij}$, where $w_{ij}=0$ and $i\neq j$, as well as the total 
network communicability and the Perron network communicability of the graph obtained when 
changing $w_{ij}$ from 0 to 1. The table shows both the total network communicability and 
the Perron network communicability to increase the most by adding an edge with either the 
largest total network sensitivity $S^{\rm TN}_{ij}$ or the largest Perron network 
sensitivity $S^{\rm PN}_{ij}$. The table suggests that the two busiest airports ($224$, 
JFK airport in NY, and $257$, LGA airport in NY), be connected. This can be done by adding
a shuttle bus between them. 

The CPU time required for computing the Perron root sensitivity, the Perron network 
sensitivity, the exact solution, and the five approximated solutions, together with the 
average number of steps that each iterative method carried out, 
are reported in Table \ref{table14}. Both the Perron network sensitivity and the Perron 
root sensitivity suggest the insertion of the same edge to maximize the Perron network 
communicability, but computing the Perron root sensitivity of $A$ is much cheaper than 
computing the Perron network sensitivity for large networks, as shown in Table \ref{table14}. 
While computing the total network sensitivity, reorthogonalization is carried out for the 
Arnoldi method \eqref{approx1} and rebiorthogonalization is performed for the Lanczos 
biorthogonalization method \eqref{lbapprox1}. 

For this example, the approximations in the stopping criterion \eqref{reldiff} are 
large. Therefore, the Lanczos  biorthogonalization method \eqref{lbapprox1} 
requires many steps to satisfy this criterion when $r_{ij}<10^{-4}$. This makes application of 
the method expensive.  We therefore replaced the matrix $M$ by $M/3$. Then the Lanczos 
biorthogonalization method \eqref{lbapprox1} requires fewer steps to satisfy the analogue 
of \eqref{reldiff} with $r_{ij}<10^{-4}$. This modification showed the correct edge to insert. 
Making the tolerance for $r_{ij}$ smaller than $10^{-4}$ to determine when to 
terminate the iterative methods did not change the edge selected.  

Figure \ref{fig8} is analogous to Figure \ref{fig7} and shows that 
the approximate solution obtained using KKRS Arnoldi method to be accurate, followed 
by the one obtained by Lanczos biorthogonalization method \eqref{lbapprox3}, but the
latter is faster and requires a smaller average number of steps. The solutions obtained by 
the Arnoldi method \eqref{approx2} can be seen to be less accurate than those obtained 
with the KKRS Arnoldi method, and the Lanczos biorthogonalization method \eqref{lbapprox3} (the
dots are not as close to a straight line). The solution obtained by the Arnoldi \eqref{approx1} and 
the Lanczos biorthogonalization methods \eqref{lbapprox1} are the least accurate ones compared to 
those obtained by the other three iterative methods. On the other hand, each step of the Arnoldi
method \eqref{approx1} requires one matrix-vector product evaluation with $M$ (two 
matrix-vector product evaluations with $A$), and each step of the KKRS Arnoldi method 
demands two matrix-vector product evaluations, one with $A$, and one with $A^T$, while 
each step of the Arnoldi method \eqref{approx2} only requires one matrix-vector product 
evaluation with $A$. The Lanczos biorthogonalization method \eqref{lbapprox3} is 
competitive, each step of which requires one matrix-vector product evaluation with $A$ and one with $A^T$.
  
Figure \ref{fig9} displays the exact solution and approximate solutions 
determined by the Arnoldi method \eqref{approx2} by choosing different stopping criteria. The 
method needs $11.8587$ average number of steps ($25.66$ mins) to satisfy the 
stopping criterion $r_{ij}<10^{-5}$. The figure shows that the approximation error is smaller when choosing 
$r_{ij}<10^{-5}$ and the number of matrix-vector product evaluations is still the smallest one.
Table \ref{table15} displays the airport labels
corresponding to the airports of Table \ref{table13}.  $~~\Box$

\begin{table}[H]
\renewcommand{\arraystretch}{1.15}
\begin{center}
\begin{tabular}{cc}
\hline  
$\{i,j\}$ & $S^{\rm PR}_{ij}$  \\
\hline
$\{224,257\}$   &  0.01341        \\
$\{257,224\}$   &  0.01335        \\
$\{287,224\}$   &  0.01240         \\
$\{261,24\}$     &  0.01230          \\
$\{224,287\}$   &  0.01228          \\
\hline
\end{tabular}\\
\vspace{0.05in}
\end{center} 
\caption{Example \ref{ex8}: The five largest Perron root sensitivities along directions 
for which $w_{ij}=0$ and $i\neq j$.}
\label{tableev2}
\end{table}

\begin{table}[H]
\renewcommand{\arraystretch}{1.15}
\begin{center}
\begin{tabular}{ccccccc}
\hline  
$\{i,j\}$ & $S^{\rm TN}_{ij}$ & $C^{\rm TN}(A+E_{ij})$ & $\{i,j\}$ & $S^{\rm PN}_{ij}$ & 
$C^{\rm PN}(A+E_{ij})$   \\
\hline
$ \{224,257\}$  & $2.5260\times 10^{36}$   & $1.9418\times 10^{38}$   & $ \{224,257\}$ & $2.5237\times 10^{36}$ &  $1.9386\times 10^{38}$ \\
$ \{257,224\}$  &$2.5161\times 10^{36}$    & $1.9417\times 10^{38}$   & $ \{257,224\}$ & $2.5103\times 10^{36}$  & $1.9385\times10^{38}$\\
$ \{261,24\}$    & $2.3817\times 10^{36} $  & $1.9404\times 10^{38}$   & $ \{261,24\}$   & $2.3810\times 10^{36}$ &  $1.9372\times 10^{38}$\\
$ \{19,124 \}$   & $2.3576\times 10^{36}$   & $1.9401\times 10^{38}$   & $ \{19,124\}$   & $2.3587\times 10^{36}$ &  $1.9369\times 10^{38}$\\
$ \{24, 261\}$   & $2.3529\times 10^{36}$   & $1.9401\times 10^{38}$   & $ \{24,261\}$   & $2.3457\times 10^{36}$ &  $1.9368\times 10^{38}$ \\
\hline
\end{tabular}\\
\vspace{0.05in}
\end{center} 
\caption{Example \ref{ex9}: The five largest total network sensitivities and Perron 
network sensitivities to changes in $w_{ij}$, where $w_{ij}=0$ and $i\neq j$, together 
with the corresponding total network communicability and Perron network 
communicability when $w_{ij}$ is increased from 0 to 1.}
\label{table13}
\end{table}
 
\begin{table}[H]
\renewcommand{\arraystretch}{1.15}
\begin{center}
\begin{tabular}{lcccc}
\hline  
\textbf{Communicability} &\textbf{Method} & \textbf{Average steps}& \textbf{Elapsed time (in seconds)} \\
\hline 
Perron network communicability     &  Perron root sensitivity                       & N/A &   0.65 \\
                                                          &  Perron network sensitivity               & N/A &   9228 (2.56 hrs) \\
  \hline
                                                          &  Exact solution                                  & N/A  &    500172 (139hrs) \\
                                                          &   Arnoldi solution \eqref{approx1}      &  16.9994   &   2070(34.5 mins)\\
Total network communicability          &   Arnoldi solution \eqref{approx2}      &  10.2312    &    1102 (18.4 mins)  \\
                                                          &  Lanczos solution \eqref{lbapprox1}  &  40.1211   &  26381 (7.3 hrs)    \\
                                                          &  Lanczos solution \eqref{lbapprox3}  &  8.0560    &  934 ( 15.6 mins)\\
                                                          &  KKRS Arnoldi solution                      &  9.9561    &  1521 (25.4 mins) \\
\hline
\end{tabular}\\
\vspace{0.05in}
\end{center} 
\caption{CPU time for the evaluation of the Perron root sensitivity, the Perron network 
sensitivity, the exact solution, and the five approximated solutions, as well as the 
average number of steps $\overline{m}$ that each iterative method needed to satisfy the stopping criterion.}
\label{table14}
\end{table}

\begin{table}[H]
\renewcommand{\arraystretch}{1.15}
\begin{center}
\begin{tabular}{cc}
\hline  
\textbf{Label} & \textbf{Airport} \\
\hline
19  &  AMS airport in NL \\
24  &  ATL airport in GA, USA \\
124 & DFW airport in TX, USA \\
224 & JFK airport in NY, USA \\
257 & LGA airport in NY, USA \\
261 & LHR airport in UK \\
\hline
\end{tabular}\\
\vspace{0.05in}
\end{center} 
\caption{Airports for airport labels shown in Table \ref{table13}.}
\label{table15}
\end{table}

\begin{figure}[H]
\begin{center}
\includegraphics[scale=0.42]{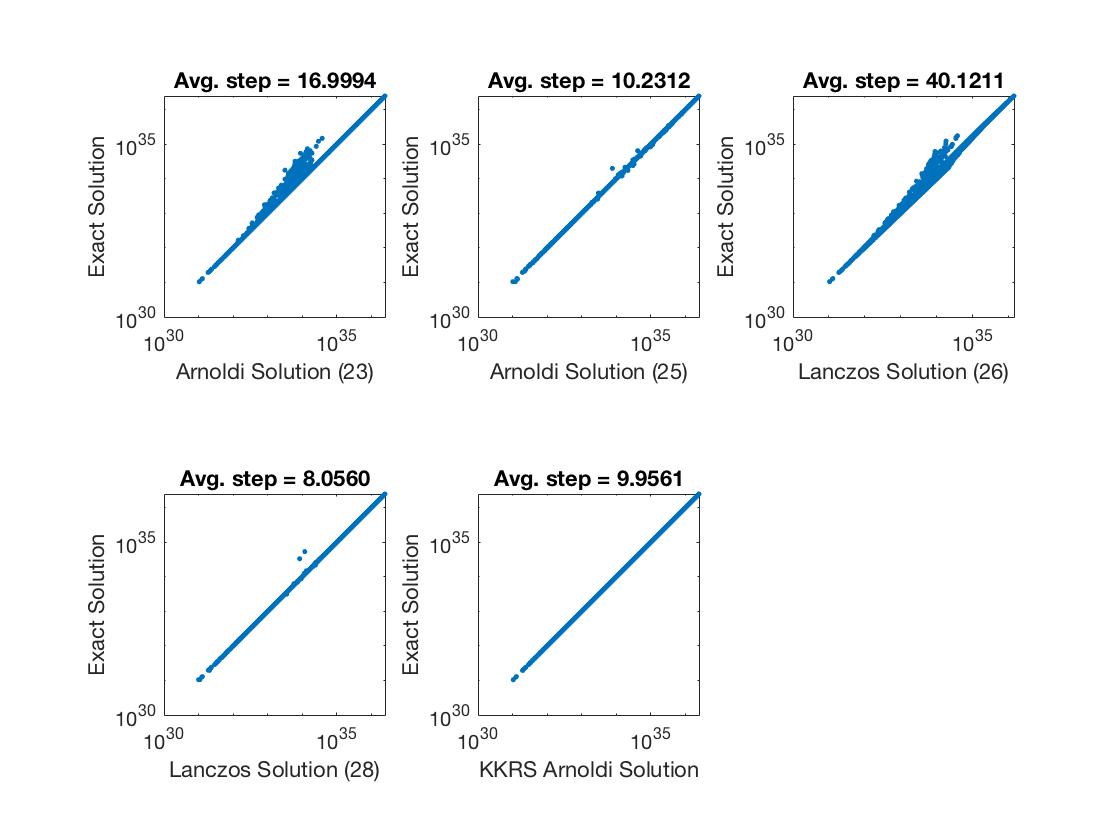}
\caption{Air500: Comparison of the exact and approximate solutions using the five 
iterative Krylov subspace methods described in Section \ref{lsn}. The heading of each 
subplot is the average number of steps $\overline{m}$ that each iterative method demands to satisfy the 
stopping criterion.}
\label{fig8}
\end{center}
\end{figure}

\begin{figure}[H]
\begin{center}
\includegraphics[scale=0.42]{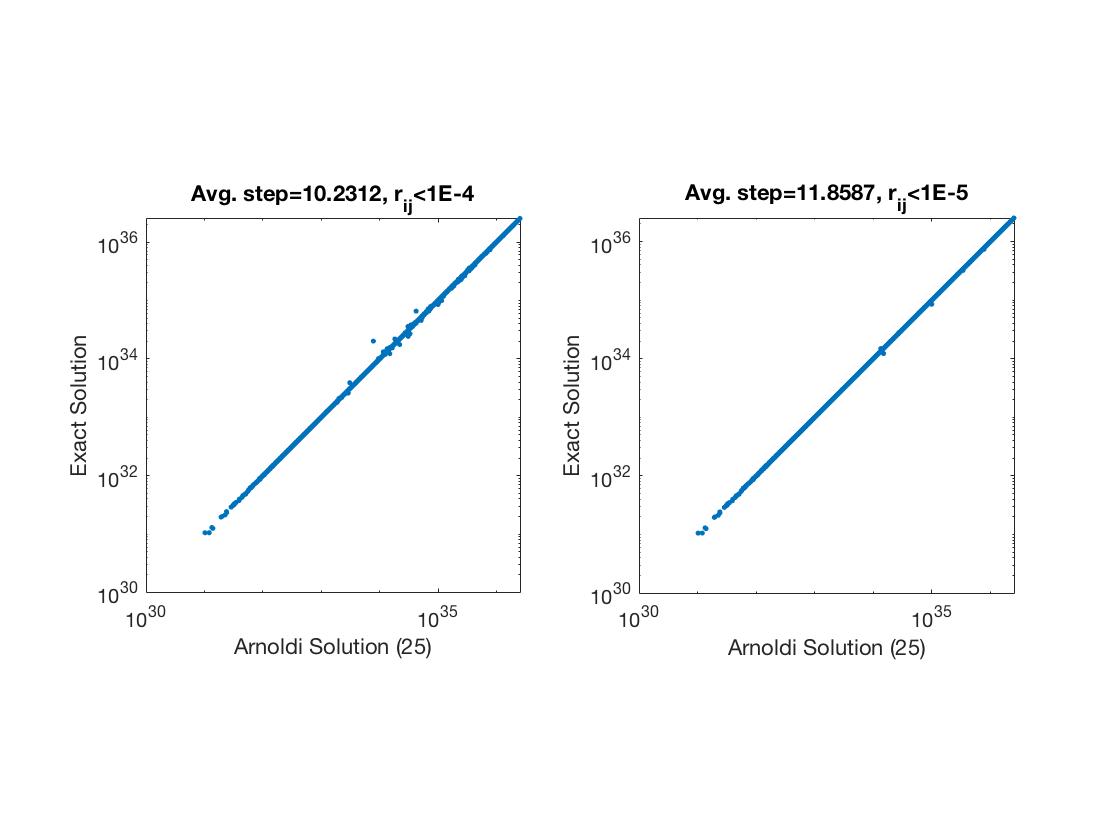}
\caption{Air500: Comparison of the exact and approximate solutions using the Arnoldi 
method \eqref{approx2} by choosing different stopping criterions. The heading of each 
subplot is the chosen stopping criterion and the average number of steps $\overline{m}$ 
that the method demands to satisfy the corresponding stopping criterion.}
\label{fig9}
\end{center}
\end{figure}

\label{ex9}
\end{example}

\begin{example}
We consider the network usroads-48, which corresponds to the continental US road network, 
and is represented by an unweighted undirected graph. The graph has $126146$ nodes, which 
correspond to intersections and road endpoints. Undirected edges represent roads which 
connect the intersections and endpoints. Our aim is to insert an edge 
$e(v_i~\LeftArrow{0.25cm}\RightArrow{0.34cm}~v_j)\notin\mathcal{E}$ such that the network 
communicability is increased the most. Results of Examples \ref{ex8} and \ref{ex9} show 
that computing the Perron root sensitivity of the modified adjacency matrix $\hat{A}$ or 
adjacency matrix $A$ is much cheaper than computing the Perron network sensitivity and 
total network sensitivity for large networks. Also note that it is not straightforward to 
determine whether a large network is irreducible. We therefore add a perturbation matrix 
$\delta A = \delta\cdot\mathbf{1}\mathbf{1}^T$ of small norm to the adjacency matrix $A$. 
This guarantees that the matrix considered is irreducible. To determine the value of 
$\delta$, we use the same procedure as in Example \ref{ex_new}. While decreasing the value
of $\delta$, we use the available computed Perron vectors for the previous $\delta$-value as 
initial approximations to compute new Perron vectors associated with the new smaller value of
$\delta$. We obtain in this manner the irreducible matrix $\hat{A}=A+\delta A$ with 
$\delta = 10^{-7}$. Note that the matrix $\hat{A}$ is not explicitly formed; we evaluate
matrix-vector products with $\hat{A}$ and $\hat{A}^T$ without explicitly storing these
matrices. This keeps both the storage and computing time low.

Table \ref{tableev3} displays the five largest Perron root sensitivities $S^{\rm PR}_{ij}$ 
and shows that the Perron root is increased the most by inserting the edge 
$e(v_{44182}\leftrightarrow v_{44035})$, that is, building a new road between endpoints 
44182 and 44035. The network of the present example is too large to make the evaluation
of total network sensitivities and the exact solution practical. $~~\Box$

\begin{table}[H]
\renewcommand{\arraystretch}{1.15}
\begin{center}
\begin{tabular}{cc}
\hline  
$\{i,j\}$ & $S^{\rm PR}_{ij}$  \\
\hline
$\{44182,44035\}$    &  0.0898        \\
$\{44067,44323\}$    &  0.0846        \\
$\{44154,44087\}$    &  0.0845         \\
$\{44182,44133\}$    &  0.0797          \\
$\{44182,44294\}$    &  0.0795          \\
\hline
\end{tabular}\\
\vspace{0.05in}
\end{center} 
\caption{Example \ref{ex10}: The five largest Perron root sensitivities along directions 
for which $w_{ij}=0$ and $i\neq j$.}
\label{tableev3}
\end{table}

\label{ex10}
\end{example}

The examples of this section, as well as numerous other numerical experiments, indicate 
that the Arnoldi-based method \eqref{approx2}, the KKRS Arnoldi method, and the Lanczos 
biorthogonalization method \eqref{lbapprox3} to perform better than the Arnoldi-based 
method \eqref{approx1} and the Lanczos biorthogonalization method \eqref{lbapprox1}. The 
Arnoldi method \eqref{approx2} typically requires fewer matrix-vector product evaluations 
and less computer storage than the KKRS Arnoldi and Lanczos biorthogonalization methods 
\eqref{lbapprox3}. This suggests that the former method may be attractive to use to 
compute the total network sensitivities $S^{\rm TN}_{ij}$ when the evaluation of 
matrix-vector products is very expensive or computer storage is scarce, otherwise the KKRS
Arnoldi method should be applied. For large networks, one can compute the Perron root 
sensitivity. Its computation is much cheaper than the evaluation of the Perron network 
sensitivity and total network sensitivity. The latter measures therefore are less attractive
to use for large networks.

\section{Conclusion}\label{conclusion}
This paper explores the sensitivity of global communicability measures to small local changes in
a network. In particular, we investigate the sensitivity of the network communicability by
increasing or decreasing edge-weights, or adding or removing edges such that the total 
network communicability or Perron network communicability significantly increases or 
decreases. The latter communicability measure is new and has the advantage of being easy
to apply to very large networks. Efficient Krylov subspace-type methods for estimating the 
total network sensitivity, the Perron root sensitivity, and the Perron network 
sensitivity are introduced and compared. Computed examples illustrate the feasibility of 
the methods described.

\section*{Acknowledgments}
The authors would like to thank Ernesto Estrada, Marcel Schweitzer, and the referees for 
comments. Work of ODCC and LR was supported in part by NSF grant DMS-1720259. Work by SN 
was partially supported by a grant from SAPIENZA Universit\`a di Roma and by INdAM-GNCS.

\bibliographystyle{plain}

\end{document}